\documentclass[12pt,reqno]{amsart}
\usepackage{amsmath,amsfonts,amssymb,amsthm,amscd,latexsym}

\usepackage{tikz}
\usetikzlibrary{arrows}
\usepackage{color}
\usepackage{multirow}

\newtheorem{thm}{Theorem}[section]
\newtheorem{cor}[thm]{Corollary}
\newtheorem{lm}[thm]{Lemma}
\newtheorem{prop}[thm]{Proposition}
\newtheorem{defn} [thm]{Definition}
\newtheorem{rem}[thm]{Remark}

\numberwithin{equation}{section}

\newcommand{\tr}{\mathop{tr}}
\newcommand{\cM}{{\mathcal M}}
\newcommand{\cN}{{\mathcal N}}
\newcommand{\cA}{{\mathcal A}}

\begin{document}

\date{\today}

\title[Notes on derivations of Murray--von Neumann algebras]{Notes on derivations of Murray--von Neumann algebras}

\author[Ber]{Aleksey Ber}
\address{ Department of Mathematics, National University of Uzbekistan, Vuzgorodok,
100174, Tashkent, Uzbekistan}
\email{aber1960@mail.ru}

\author[Kudaybergenov]{Karimbergen Kudaybergenov}
\address{ Department of Mathematics, Karakalpak State University, Ch. Abdirov 1, Nukus 230113, Uzbekistan}
\email{karim2006@mail.ru}

\author[Sukochev]{Fedor Sukochev}
\address{ School of Mathematics and Statistics, University of New South Wales,
Kensington, 2052, Australia}
\email{f.sukochev@unsw.edu.au}

\begin{abstract}
Let $\mathcal{M}$ be  a type II$_1$ von Neumann factor and let $S(\mathcal{M})$ be the associated Murray-von Neumann algebra of all measurable operators affiliated to $\mathcal{M}.$   We extend a result of Kadison and Liu \cite{KL} by showing that any derivation from $S(\mathcal{M})$ into an $\mathcal{M}$-bimodule $\mathcal{B}\subsetneq S(\mathcal{M})$ is trivial. In the special case, when $\mathcal{M}$ is the hyperfinite type  II$_1-$factor $\mathcal{R}$, we introduce the algebra $AD(\mathcal{R})$, a noncommutative analogue of the algebra of all almost everywhere approximately differentiable functions on $[0,1]$ and show that it is a proper subalgebra of $S(\mathcal{R})$. This algebra is strictly larger than the corresponding ring of continuous geometry introduced by von Neumann. Further, we establish that the classical approximate derivative on (classes of) Lebesgue measurable functions on $[0,1]$ admits an extension to a derivation from $AD(\mathcal{R})$ into $S(\mathcal{R})$, which fails to be spatial. Finally, we show that for a Cartan masa $\mathcal{A}$ in a hyperfinite II$_1-$factor $\mathcal{R}$ there exists a derivation $\delta$ from $\mathcal{A}$ into $S(\mathcal{A})$ which does not admit an extension up to a derivation from $\mathcal{R}$ to $S(\mathcal{R}).$

\end{abstract}

\dedicatory{To the memory of Richard Kadison}

\maketitle

\section{Introduction}\label{Intro}

Let
$\mathcal{A}$ be an algebra over the field of complex numbers and $\mathcal{B}$ be an $\mathcal{A}$ -bimodule. A linear operator $D:\mathcal{A}\rightarrow \mathcal{B}$
is called a \textit{derivation} if it satisfies
the identity $D(xy)=D(x)y+xD(y)$ for all  $x, y\in \mathcal{A}$. Each element  $a\in \mathcal{B}$ defines a linear
derivation $\textrm{ad}_a:D:\mathcal{A}\rightarrow \mathcal{B}$ given by $\textrm{ad}_a(x)=ax-xa,\,x\in
\mathcal{A}$.  Such derivations $\textrm{ad}_a$ are called \textit{spatial derivations}. If the element  $a$ implementing the derivation
$\textrm{ad}_a$ belongs to $\mathcal{A},$ then $\textrm{ad}_a$ obviously maps $\mathcal{A}$ into itself and is called  \textit{inner derivation} (of the algebra $\mathcal{A}$).

The theory of derivations in operator algebras is an important and well studied part of the general theory
 of operator algebras, with applications in mathematical physics (see, e.g.~\cite{Bra}, \cite{Sak2}). It is well known that every derivation
of a
  $C^{\ast}$-algebra is bounded (i.e. is norm continuous), and that every derivation of a von
 Neumann algebra is  inner.
 For  a detailed exposition of the theory of bounded derivations we refer to the monograph of Sakai~\cite{Sak2}.

The development of a non-commutative integration theory was initiated by Segal \cite{Seg.}, who introduced new classes of (not necessarily Banach) algebras of unbounded operators, in particular the algebra  $S(\mathcal{M})$ of all measurable operators affiliated with a von Neumann algebra $\mathcal{M}$ (see next section for precise definitions).

The properties of derivations of the
algebra $S(\mathcal{M})$ are far from being similar to those exhibited
by derivations on von Neumann algebras.
On one hand, for commutative von Neumann algebra $\mathcal{M}=L_\infty [0,1]$, the algebra $S(\mathcal{M})$  coincides with Lebesgue space  $S[0, 1]$ of all  measurable complex functions on
the interval $[0, 1]$, and the latter algebra admits non trivial (and hence, non-inner) derivations \cite{BSCh04, BSCh06}.
On the other hand, if $\mathcal{M}$ is a properly infinite von Neumann algebra, then all derivations on $S(\mathcal{M})$ are inner (\cite[Theorem 2.7]{AAK}, \cite[Theorem 4.17 and Proposition 5.3]{BdPS}, \cite[Corollary 5.1]{BSCh13} and \cite[Corollary 4.2]{BCS14}).
These two dramatically different results indicate a special interest (and difficulty) in the case when $\mathcal{M}$ is a type II$_1-$von Neumann algebra, and this is precisely the case in which we are interested in this paper. In this case, $S(\mathcal{M})$ is the algebra of all operators affiliated with $\mathcal{M}$, which is sometimes referred to as the Murray-von Neumann algebra associated with $\mathcal{M}$ (see e.g. \cite{KL}). It is still unknown whether the algebra $S(\mathcal{M})$ admits non-inner derivations. To our best knowledge, the question whether every derivation on $S(\mathcal{M})$ is necessarily inner was firstly posed in \cite{Ayupov}.
A partial step towards proving that $S(\mathcal{M})$ may not admit any non-inner derivations was made by  Kadison and Liu \cite{KL} who showed that any derivation from $S(\mathcal{M})$ into $\mathcal{M}$ is necessarily trivial when $\mathcal{M}$ is a von Neumann algebra of type II$_1$.
In fact, it is conjectured in \cite[p.211]{KL} that  $S(\mathcal{M})$ does not admit non-inner derivations in this setting. In this paper, we partially confirm this conjecture by showing that any derivation from the Murray-von Neumann algebra $S(\mathcal{M})$ associated with any type II$_1$ von Neumann algebra $\mathcal{M}$,  with values in a Calkin operator space $\mathfrak{B}\subsetneq S(\mathcal{M})$ is necessarily trivial (see Theorem~\ref{KL-new}).

The result of \cite{KL} cited above corresponds to the very special case $\mathfrak{B}=\mathcal{M}$. It is worthwhile to point out that if $\mathcal{M}$ is a type II$_1-$ factor, then every $\mathcal{M}$-bimodule $\mathfrak{B}\subseteq S(\mathcal{M})$ is automatically a Calkin operator space. In other words in this special case our result states that every derivation from  $S(\mathcal{M})$ into any $\mathcal{M}$-bimodule  distinct from $S(\mathcal{M})$ is trivial  (see Corollary~\ref{cor_ext_KL}). Our proof is based on an entirely different approach to that of \cite{KL}, and appears to be of interest in its own right.

The second part of the paper is concerned with extensions of derivations initially defined on abelian subalgebras $\mathcal{A}$, of a type II$_1$ von Neumann algebra $\mathcal{M}$.  Here, we concentrate on the special case where $\mathcal{M}$ coincides with the hyperfinite type II$_1$ factor $\mathcal{R}$, and $\mathcal{A}$ coincides with a special Cartan masa in $\mathcal{R}$, the ``diagonal'' subalgebra $\mathcal{D}$ of $\mathcal{R}$.  The algebra $\mathcal{D}$ is $*$-isomorphic to the algebra $L_\infty[0,1]$, and therefore, there exists a $*$-subalgebra $AD(\mathcal{D})\subset S(\mathcal{R})$, which is $*$-isomorphic to the classical $*$-subalgebra of all  almost everywhere approximately differentiable function of $S[0,1]$ (see Section \ref{sec_AD_intro} for precise definitions). Next, we construct a noncommutative analogue $AD(\mathcal{R})$, generalising the algebra $AD(\mathcal{D})$ of \lq\lq approximately differentiable operators\rq\rq in $S(\mathcal{R})$, and show that this algebra admits a derivation, which extends the approximate derivation on $AD(\mathcal{D})$ (see Theorem \ref{b main thm}).
The $\ast$-algebra $AD(\mathcal{R})$ contains as a proper $\ast$-subalgebra  the regular ring  $C_\infty$ of continuous geometry for $\mathbb{C},$
constructed by J. von Neumann as a completion  in the rank-metric of a sum of an increasing sequence of matrix rings over the field of complex numbers. Continuous geometry was developed by J. von Neumann in the period  1935-37 in his series of article consisting of five papers (see e.g. \cite{Neugeo, Neu62}). In particular, the notion of rank-distance was firstly defined in \cite{Neu37} (see also \cite[pp. 160-161]{Neu62}), and described by von Neumann as \lq\lq a really significant topology\rq\rq (see \cite[p. 137]{Neu62}). This topology also plays a crucial role in our construction of the algebra $AD(\mathcal{R}).$  It is of interest to observe that the properties of this topology also play an important role in our extension of the Kadison-Liu result from \cite{KL}, described beforehand.

Finally, in the last section of this paper we show that there exists a derivation $\delta$ from a Cartan masa $\mathcal{A}$ of $\mathcal{R}$ with values in $S(\mathcal{A})$, which cannot be extended to a derivation from $\mathcal{R}\to S(\mathcal{R})$. This derivation $\delta$ is nothing fancy, in fact it is a twisted version of the approximate derivation on $AD(\mathcal{D})$  which fails to have an extension  up to a derivation on $S(\mathcal{R})$ (see Theorem \ref{thm_nonext_diagonal}). The crucial result in this proof is \cite[Theorem 1.2]{BSZ} (restated below as Theorem \ref{com}), which states that the identity of the algebra $\mathcal{M}$ can not be written as commutator  $[a,b]$ with $a,b\in S(\mathcal{M})$ if one of the elements $a$ or $b$ is normal.

The paper is organized as follows. In Section \ref{sec2} we gather necessary preliminaries. Section \ref{sec6} is devoted to the Kadison-Liu conjecture.

In Section~\ref{sect1} we prove that the largest subalgebra of $S[0,1]$ which admits a unique extension of the classical derivation $\frac{d}{dt}$ on $[0,1]$ is the algebra $AD[0,1]$ of all approximately differentiable functions.

In Section~\ref{der_on_factors}, for a  hyperfinite factor $\mathcal{R}$ of type II$_1$, we construct a dense (with respect to the measure topology) $\ast$-regular (in the sense von Neumann) subalgebra $AD(\mathcal{R})$ in $S(\mathcal{R}).$ This algebra contains  a $*$-subalgebra  $\ast$-isomorphic to
the algebra $AD[0,1]$ and can be viewed as a noncommutative analogue of approximately differentiable functions. We prove that the approximate derivative on $\partial_{AD}:AD[0,1]\to S[0,1]$ can be extended up to a derivation $\delta_{AD}:AD(\mathcal{R})\to S(\mathcal{R})$, and that this derivation is not spatial.

In Section~\ref{sec5} we show that a twisted version of the approximate derivative on the algebra $S(\mathcal{D})$ cannot be extended up to a derivation on the whole algebra $S(\mathcal{R})$, and prove a similar result for an arbitrary
Cartan masa $\mathcal{A}$ in the hyperfinite II$_1-$factor $\mathcal{R}$.

\subsection*{Acknowledgement} The authors thank Dmitriy Zanin and Galina Levitina for useful discussions and comments on earlier versions of the present paper and Jinghao Huang and Thomas Scheckter for careful reading of the manuscript and supplying useful feedback. We also thank Kenneth Dykema for discussion of Cartan subalgebras in the hyperfinite II$_1$-factor. Some results of Section~\ref{sect1} were presented by the first named author at Crimea Autumn Mathematical School KROMSH-{2005}.

\section{Preliminaries}\label{sec2}

In this section we briefly list some necessary facts concerning algebras of measurable operators.

Let $H$  be a Hilbert space and let $B(H)$ be the $\ast$--algebra of all bounded linear operators
on $H.$ A von Neumann algebra $\mathcal{M}$  is a weakly closed unital $\ast$-subalgebra in $B(H).$ For details on von Neumann algebra theory, the reader is
 referred to \cite{Dixmier, KR1, KRII, SS, Tak}. General facts concerning measurable operators may be found in \cite{Nel,Seg.} (see also \cite[Chapter IX]{Tak2} and the forthcoming book \cite{DPS}). For convenience of the reader, some of the basic definitions are recalled below.

\subsection{The Murray-von Neumann algebra}

A densely defined closed linear operator $x : \textrm{dom}(x) \to  H$
(here the domain $\textrm{dom}(x)$ of $x$ is a linear subspace in $H$) is said to be \textit{affiliated} with $\mathcal{M}$
if $yx \subset  xy$ for all $y$ from the commutant $\mathcal{M}'$  of the algebra $\mathcal{M}.$

Recall that two projections $e, f \in  \mathcal{M}$ are called \textit{equivalent} if there exists an element
$u \in \mathcal{M}$ such that $u^\ast  u = e$ and $u u^\ast  = f.$
A projection $p \in \mathcal{M}$ is called \textit{finite}, if the
conditions $q \leq  p$ and $q$ is equivalent to $p$ (denoted by $p\sim q$)  imply that $q = p.$
A linear operator $x$ affiliated with $\mathcal{M}$ is called \textit{measurable} with respect to $\mathcal{M}$ if
$\chi_{(\lambda,\infty)}(|x|)$ is a finite projection for some $\lambda>0.$ Here
$\chi_{(\lambda,\infty)}(|x|)$ is the  spectral projection of $|x|$ corresponding to the interval $(\lambda, +\infty).$
We denote the set of all measurable operators by $S(\mathcal{M}).$
Clearly, $\mathcal{M}$  is a subset of $S(\mathcal{M}).$

Let $x, y \in  S(\mathcal{M}).$ It is well known that $x+y$ and
$xy$ are densely-defined and preclosed
operators. Moreover, the (closures of) operators $x + y, xy$ and $x^\ast$  are also in $S(\mathcal{M}).$
When
equipped with these operations, $S(\mathcal{M})$ becomes a unital $\ast$-algebra over $\mathbb{C}$  (see [19]). It
is clear that $\mathcal{M}$  is a $\ast$-subalgebra of $S(\mathcal{M}).$

For a self-adjoint $x\in S(\mathcal{M})$ we denote by $x_+$ (respectively, $x_-$) its posiitve (respectively negative part), defined by $x_+=\frac{x+|x|}{2}$ (respectively, $x_-=\frac{x-|x|}{2}$). We note that $x_-$ and $x_+$ are orthogonal, that is $x_-x_+=0$.

If, for example, if $\mathcal{M}$ is finite, then every operator affiliated with $\mathcal{M}$ becomes measurable.
In particular, the set of all affiliated operators forms a $\ast$-algebra, which coincides with
$S(\mathcal{M}).$
Following \cite{KL, KLT}, in the case when von Nemaunn algebra $\mathcal{M}$ is finite, we refer to the algebra $S(\mathcal{M})$ as the Murray-von Neumann algebra associated with $\mathcal{M}$.

Let $\tau$ be a faithful normal finite trace on $\mathcal{M}.$
Consider the topology  $t_{\tau}$ of convergence in measure or \textit{measure topology}
on $S(\mathcal{M}),$ which is defined by
the following neighborhoods of zero:
$$
N(\varepsilon, \delta)=\{x\in S(\mathcal{M}): \exists \, e\in P(\mathcal{M}), \, \tau(\mathbf{1}-e)\leq\delta, \, xe\in
\mathcal{M}, \, \|xe\|_\infty\leq\varepsilon\},
$$
where $\varepsilon, \delta$
are positive numbers,    $\mathbf{1}$ is the unit in $\mathcal{M}$ and $\|\cdot\|_\infty$ denotes the operator norm on
$\mathcal{M}.$  The algebra $S(\mathcal{M})$ equipped with the measure topology is a topological algebra.

We also recall the following result (see e.g. \cite[Proposition 3.3]{dePS_flow})
\begin{prop}\label{prop_trace_pres_isom}
Let $\mathcal{M}$ and $\mathcal{N}$ be von Neumann algebras equipped with faithful normal finite traces. If $\alpha:\mathcal{M}\to \mathcal{N}$ is a $*$-isomorphisms which preserves the trace. Then, $\alpha$ extends up to a $*$-isomorphism of $S(\mathcal{M})$ and $S(\mathcal{N})$, which is also continuous in the measure topology.
\end{prop}

If  $m$ denotes Lebesgue measure on the interval $[0,1]$, and
if we consider $L^\infty(m)$ as an Abelian von Neumann algebra
acting via multiplication on the Hilbert space   ${\mathcal
H}=L^2(m)$, with the trace given by integration with respect to
$m$, then $S(L^\infty(m)) $ consists of all measurable functions on
$[0,1]$ which are bounded except on a set of finite measure.
In other words, the algebra $S(L^\infty(m))$ coincides with the space $S(0,1)$ of all a.e. finite Lebesgue measurable functions on $[0,1]$ (and we will keep the later notation for this algebra) and convergence for the measure
topology coincides with the usual notion of convergence in
measure.

It was established in \cite{BSZ} that the Heisenberg relation $[a, b] = \mathbf{1}$ does not hold
in the algebra of locally measurable operators affiliated with
an arbitrary infinite von Neumann algebra. In the case when
the von Neumann algebra is finite, it is proved there that $[a, b] \neq \mathbf{1}$
provided that $a$  is normal. For convenience of further referencing we state it in full.

\begin{thm}\label{com} \cite[Theorem 1.2]{BSZ}
Let $\mathcal{M}$ be a von Neumann algebra and let $a, b \in  S(\mathcal{M}).$
\begin{itemize}
\item[(a)] If $\mathcal{M}$ is infinite, then $[a, b] \neq  \textbf{1}.$
\item[(b)] If $\mathcal{M}$ is a finite type I algebra, then $[a, b] \neq \textbf{1}.$
\item[(c)] If $a$ is normal, then $[a, b] \neq \textbf{1}.$
\end{itemize}
\end{thm}

\subsection{Regular $*$-algebras and regularity of the algebra $S(\mathcal{M})$}

Let  $\mathcal{M}$ be a von Neumann algebra with a faithful
normal finite trace  $\tau$ and  let $S(\mathcal{M})$ be the Murrey-von Neumann algebra associated with $\mathcal{M}$.

A $\ast$-subalgebra  $\mathcal{A}$ of $S(\mathcal{M})$ is said to be
\emph{regular}, if it is a regular ring in the sense of von Neumann, i.e., if for every
$a\in\mathcal{A}$ there exists an element  $b\in\mathcal{A}$ such that
$aba=a$ and  $c^*c=0$ implies $c=0$ for all  $c\in\mathcal{A}$ (see e.g. \cite{Skor}).

Let $a\in S(\mathcal{M})$ and let  $a=v|a|$ be the polar decomposition of $a.$
Then $l(a) =v v^\ast$ and  $r(a)=v^\ast v$ are left and right supports of the element  $a$, respectively.
The projection  $s(a)=l(a)\vee r(a)$ is the support of the element $a$. It is clear that $r(a)=s(|a|)$ and  $l(a)=s(|a^*|)$.

Let $|a|=\int\limits_0^\infty \lambda d e_\lambda$ be the spectral resolution of the element
$|a|\in S(\mathcal{M})$. Since $\mathcal{M}$ is finite, there exists an element
$
i(|a|)=\int\limits_0^\infty \lambda^{-1} d e_\lambda\in S(\mathcal{M})$. Moreover,
$$
s(i(|a|))=s(|a|),\ |a|i(|a|)=i(|a|)|a|=s(|a|).
$$
Set  $i(a)=i(|a|)v^*$. We have
$$
ai(a)=v|a|i(|a|)v^*=vs(|a|)v^*=l(a),\ i(a)a=i(|a|)v^*v|a|=r(a),\ ai(a)a=a.
$$
Therefore  $S(\mathcal{M})$ is a regular $*$-algebra. The element
$i(a)$ is called a \emph{partial inverse} of the element $a,$  it is a unique element in
$S(\mathcal{M})$, such that $i(a)l(a)=i(a)$ and  $i(a)a=r(a)$ (see \cite[Proposition  91]{Skor}).

Let  $\mathcal{A}$ be a  regular $*$-subalgebra of  $S(\mathcal{M})$ and
let $\mathbf{1}\in\mathcal{A}$. If  $a\in \mathcal{A}$, then  $l(a),r(a)\in\mathcal{A}$. Indeed, by \cite[Proposition   88]{Skor} the left and right ideals $\mathcal{A}a$ and $a\mathcal{A},$ are generated by projections and therefore there exist projections  $p$ and $q$ in  $\mathcal{A}$ such that
$\mathcal{A}a=\mathcal{A}q,\ a\mathcal{A}=p\mathcal{A}$. Thus
$a=bq$ for some $b\in \mathcal{A}$ and so  $aq=bqq=bq=a,$ and therefore
$r(a)\leq q$. On the other hand,  $q=ca$, hence
$q r(a)=car(a)=ca=q$, and so  $q\leq r(a)$. We conclude that  $r(a)=q\in\mathcal{A}$. Similarly, $l(a)=p\in\mathcal{A}$.

\subsection{Derivations on algebras} \label{sec_der_prel}

\begin{defn}\label{def_nonexpansive}Let $\mathcal{A},\mathcal{B}$ be $*$-subalgberas in $S(\mathcal{M})$.
A derivation $\delta:\mathcal{A}\to \mathcal{B}$ is called \emph{non-expansive} if $s(\delta(x))\leq s(x)$ for all $x\in \mathcal{A}$.
\end{defn}

We define the so-called rank  metric $\rho$ on $S(\mathcal{M})$  by setting
$$
\rho(x, y)=\tau( r((x-y)))=\tau(l(x-y)),\,\, x, y\in \mathcal{A}.
$$
 In fact, the rank-metric $\rho$ was firstly introduced in a general case of regular rings in \cite{Neu37}, where it was shown it is a metric.
By \cite[Proposition 2.1]{Ciach}, the algebra $S(\mathcal{M})$ equipped with the metric $\rho$ is a complete topological ring .

\begin{prop}\label{delta continuity lemma}
Let $\mathcal{A}$ be a $*$-regular subalgebra of $S(\mathcal{M})$. Any derivation $\delta:\mathcal{A}\to S(\mathcal{M})$ is continuous with respect to the metric $\rho$.
\end{prop}
\begin{proof} If $x\in\mathcal{A}$, then $l(x),r(x)\in \mathcal{A}$.
We have
$$\delta(x)=\delta(l(x)\cdot x\cdot r(x))=\delta(l(x))x+l(x)\delta(x)r(x)+x\cdot\delta(r(x)).$$
Therefore,  (see e.g. \cite[p. 161, $3(\eta)$]{Neu62})
\begin{eqnarray*}
\tau(r(\delta(x))) &\leq &
 \tau(r(\delta(l(x))x))+\tau(r(l(x)\delta(x)r(x)))+\tau(r(x\cdot\delta(r(x))))=\\
&=& \tau(r(\delta(l(x))x))+\tau(r(l(x)\delta(x)r(x)))+\tau(l(x\cdot\delta(r(x))))\leq 3\tau(r(x)).
\end{eqnarray*}
Hence,
$$\rho(\delta(a),\delta(b))=\tau(r(\delta(a-b)))\leq 3\tau(r(a-b))=3\rho(a,b).$$
 This completes the proof.
\end{proof}

\section{Lack of non-trivial derivations  with values in bimodules of operators affiliated with type II$_1$-algebras} \label{sec6}

In this section we consider  a symmetric bimodule (or, a Calkin operator space)  $\mathfrak{B}\subset S(\mathcal{M})$
 over an arbitrary type  II$_1$-algebra $\mathcal{M}$.

We complement the Kadison--Liu result \cite{KL} by showing that the only derivation that maps  $S(\mathcal{M})$ into any such
$\mathcal{M}$-bimodule $\mathfrak{B}$ is trivial provided that $\mathfrak{B}\neq S(\mathcal{M})$.

We start with collecting some technical tools.

In this section, we assume $\mathcal{M}$ is an atomless von Neumann algebra with a faithful, normal, normalized trace $\tau$. For every $x\in \mathcal{M},$  the generalised singular value
function $\mu(x),$  denoted $t \to  \mu(t, x)$ for $t \in  [0,1],$
is defined by the formula (see,
e.g.,  \cite{FK}, \cite{LSZ})
$$
\mu( t, x) =\inf  \{\|xp\|_\infty: p\in P(\mathcal{M}), \, \tau(1-p)\leq t\}.
$$

For a self-adjoint element $b\in S(\mathcal{M}),$ let $\lambda(b)=\lambda(\cdot,b)$ be the eigenvalue function of $b$ (also known as the spectral scale, see \cite{ArMa}, \cite{DSZ} and \cite{HiaiN}) defined by
$$\lambda(t,b)=
\begin{cases}
\mu(t,b_+),& 0<t<\tau({\rm supp}(b_+))\\
\lim_{\varepsilon\to 0^+} \lambda (1-t)-\varepsilon,b_-),& \tau({\rm supp}(b_+))\leq t<1.
\end{cases},
$$
Assume that  $\mathcal{M}=L_{\infty}(0,1)$ and
$\tau(f)=\int_{0}^1 fdm$, $f\in  \mathcal{M}$, where  $m$ is the   Lebesgue measure on $(0,1)$.
 In this case, $S(\mathcal{M})$ consists of all complex-valued  Lebesgue measurable functions $f$ on $(0,1)$, that is $S(\mathcal{M})=S(0,1)$ \cite{FK, LSZ}. In this setting, for every $f\in S(0,1)$ (respectively, for every real-valued $f\in S(0,1)$) the function $\mu(f)$ coincides with  the \emph{right-continuous equimeasurable nonincreasing rearrangement} of $|f|$ (see e.g. \cite{HiaiN}):
$$\mu(t;f)=\inf\{s\in\mathbb{R}:\ m(\{x\in X:\,|f(x)|>s\})\leqslant t\},\ \ t\in[0,1 )$$
(respectively,
$$\lambda(t;f)=\inf\{s\in\mathbb{R}:\ m(\{x\in X:\,f(x)>s\})\leqslant t\},\ \ t\in[0,1 ).)$$

 A linear subspace $\mathcal{B}$ of $S(\mathcal{M})$ is called a Calkin operator space if $x\in \mathcal{B}$ whenever $x\in S(\mathcal{M})$ and $\mu(x)\leq \mu(y)$ for some $y\in \mathcal{B}$. A Calkin function space is the term reserved for a Calkin
operator space when $\mathcal{M}=L_\infty(0,1)$ \cite[Definition 2.4.1]{LSZ}.  If $\mathcal{B}\subseteq S(\mathcal{M})$ is a Calkin operator space, then the set
$B\subset S(0,1)$ defined by
$$
B=\{x\in S(0,1): \ \exists z\in \mathcal{B} {\rm\ such \ that}\ \mu(x)=\mu(z)  \}
$$
is a Calkin function space. Vice versa, if $B$ is Calkin function space,
then
$$
\mathcal{B}=\{x\in S(\mathcal{M}): \ \exists z\in {B} {\rm\ such \ that}\ \mu(x)=\mu(z)  \}
$$
is a Calkin operator space. This provides a canonical bijection between Calkin operator
spaces and Calkin function spaces. For this results we refer the reader to \cite[Theorem 2.4.4]{LSZ}.  We claim that every non-empty Calkin operator space $\mathcal{B}\subseteq S(\mathcal{M})$ contains $\mathcal{M}$. Indeed, in view of the above, it is sufficient to prove that $L_\infty(0,1)\subseteq B$. Since $B\neq\emptyset$, there exists $0\neq x_0\in B$. Then for some $\epsilon>0$ and for some measurable set $e\subset (0,1)$ of positive measure we have $|x|\ge \epsilon\chi_{e}$. This implies that $\chi_{e}\in B$ and so, $\chi_{[0,m(e))}\in B$. The latter, implies that $\chi_{[m(e), 2m(e))}\in B$ and repeating this argument, we infer that $\chi_{[0,1)}\in B$. This implies the claim.

The following proposition extends  \cite[Proposition 3.0.3]{CK2} (see also \cite[Proposition 1.8]{CK3} and \cite{CKSa,CS}). The proof follows \cite{CS} and is given here for convenience of the reader.
\begin{prop}\label{preserves the spectral scale}
If $x=x^* \in S(\mathcal{M},\tau)$, then there exists an atomless commutative weakly closed $*$-subalgebra $\mathcal{N}$ in $\mathcal{M}$ containing the spectral family of the operator $x$, and a $*$-isomorphism $V$ acting from  $S(\mathcal{N},\tau)$ onto
$S([0,\tau(s(x))),m)$   such that $V(x) =\lambda(x)$  and  $\lambda(V(f))=\lambda(f)$ for every  $f=f^*\in S(\mathcal{N},\tau).$
\end{prop}

\begin{proof}
Let $\nabla _0$ be a countable Boolean subalgebra in $P(\cM)$ which contains all   spectral  projections $E^{x}(r,\infty)$  and $E^x(-\infty,r)$ of $x$, where $r$ is a rational number.
Let $\nabla$ be the closure of $\nabla_0$ in the measure topology.
 Then, $\nabla$ is a complete Boolean subalgebra in $P(\cM)$  and  the least upper bound in $\nabla$ for any subset $A\subset \nabla$ coincides with the least  upper bounded of $\cA$ in $P(\cM)$.
Such subalgebra are also called regular.

Let $\triangle$ be the set of all atoms in $\nabla$ and $\triangle\ne 0$.
Since $P(\cM)$ is a non-atomic Boolean algebra, for every $q\in \triangle$, there exists a commutative  non-atomic regular Boolean subalgebra $\nabla_q$ of $P(\cM_q)$ which is separable in the measure topology.

Let $B$ be the set of all $e\in P(\cM)$ for which $e(1-\sup \triangle)\in \nabla$ and $eq\in \nabla_q$ for any $q\in \triangle$.
It is clear that $B$ is a complete regular non-atomic and separable (with respect to the measure topology)
Boolean subalgebra in $P(\cM)$ which contains all the spectral projections of $x$.
Hence, there exists an isomorphism $\phi$ from $B$ on the Boolean algebra $P(L_\infty (0,\tau(s(x))))$
such that $m(\phi(e)) = \tau(e)$ for all $e\in B$ \cite{Halmos}.

Let us denote by $\cN$ the weak closure of the $*$-algebra generated by $B$, which is a  non-atomic commutative von Neumann subalgebra of $\cM$.

By Proposition \ref{prop_trace_pres_isom}, the isomorphism $\phi$ may be extended up to the $*$-isomorphism $V$ from $S(\cN,\tau)$ onto $S(0,\tau(s(x)))$ and, in addition, $V(\cN) =L_\infty (0,\tau(s(x)))$, $\lambda (t, x) = \lambda(t,  Vx)$ for all $t>0$, $x\in S(\cM,\tau)$.
\end{proof}

The following result extends  \cite[Corollary 13]{KL}.

\begin{thm}\label{KL-new}
Let  $\mathcal{M}$ be a type II$_1$ von Neumann algebra  and let $\mathcal{B}\subsetneq S(\mathcal{M})$ be a Calkin operator space. Then any derivation $\delta: S(\mathcal{M})\longrightarrow S(\mathcal{M})$, such that $\delta(S(\mathcal{M}))\subset \mathcal{B}$ vanishes.
\end{thm}

\begin{proof}
Since $\delta(x)\in \mathcal{B}$ for all $x\in S(\mathcal{M})$ and $\mathcal{B}$ is closed with respect to conjugation,  it follows that that $\delta(x)^\ast\in \mathcal{B}$ for all $x\in S(\mathcal{M}).$
Therefore,  $\displaystyle \frac{\delta(x)+\delta(x^\ast)^\ast}{2}\in \mathcal{B}$ and $\displaystyle \frac{\delta(x)-\delta(x^\ast)^\ast}{2i}\in \mathcal{B}$
for all $x\in S(\mathcal{M}).$ Thus, replacing $\delta$ with $\displaystyle \frac{\delta+\delta^\ast}{2},$ where $\delta^\ast(x)=\delta(x^\ast)^\ast,\,\, x\in S(\mathcal{M}),$ without loss of generality, we can assume that  $\delta=\delta^\ast,$ that is, $\delta(x)^\ast=\delta(x^\ast)$ for all
$x\in S(\mathcal{M}).$

Assume that $\delta\neq 0.$ If $\delta|_{\mathcal{M}}=0$, then due to $\rho$-continuity of $\delta$ (see Proposition \ref{delta continuity lemma}) and  $\rho$-density of $\mathcal{M}$ in $S(\mathcal{M}),$ we obtain that $\delta=0,$ which contradicts with the assumption $\delta\neq 0.$ So, if $\delta\neq 0$, then there exists a self-adjoint element $a\in \mathcal{M}$ such that
$\delta(a)\neq 0.$

The operator $\delta(a)$ is self-adjoint, and so, by Proposition~\ref{preserves the spectral scale}, there exists an atomless commutative weakly closed $*$-subalgebra $\mathcal{N}$ in $\mathcal{M}$ and a $*$-isomorphism $V$ acting from  $S(\mathcal{N},\tau)$ onto
$S([0,\tau(s(\delta(a)))),m)$   such that $V(\delta(a)) =\lambda(\delta(a)).$
Setting $\pi=V^{-1}$ and
$$y=\lambda(\delta(a)),$$ we have
 $$\pi(y)=\delta(a)\quad {\rm and}\quad \pi(\chi_{[0,\tau(s(\delta(a))))})=s(\delta(a)),$$
because by  construction of $V$ the support of $y$ is $[0,\tau(s(\delta(a)))).$ In particular,
$y$ is invertible in $S[0,\tau(s(\delta(a)))).$
 We claim that the assumptions $\mathcal{B}\neq S(\mathcal{M})$ and $\mathcal{B}$ is a Calkin operator space, imply that there exists $x\in S[0,\tau(s(\delta(a))))$ such that  $\pi(x)\notin \mathcal{B}.$
   Indeed, let us  consider the function Calkin space $B$ introduced above.
Since $\mathcal{B}\neq S(\mathcal{M})$, it follows that $B\neq S(0,1)$.
Now, we simply take any $z=\mu(z)\in S(0,1)$ such that $z\notin B$ and set $x=\mu(z)\cdot \chi_{[0, \tau(s(\delta(a)))}.$ Observe that $x\notin B$. Indeed, by the claim preceding Proposition~\ref{preserves the spectral scale}, we know that the bounded function $z-x\in B$ and therefore $x\notin B$. The fact that $\pi(x)\notin \mathcal{B}$ now follows immediately from the fact that $\mu(\pi(x))=\mu(x)$ established in \cite[Proposition 3.3(ii)]{dePS_flow}.
We shall now finish the proof. Take
$b=\pi(y^{-1}x).$
Then
$$
\delta(a)b=\pi(y)b=\pi(y)\pi\left(y^{-1}x\right)=\pi(x).
$$
Computing $\pi(x)=\delta(a)b=\delta(ab)-a\delta(b)\in\mathcal{B}$, we arrive at the contradiction.
\end{proof}

 Recall that a linear subspace $\mathcal{J}$ of $S(\mathcal{M})$ is called an operator bimodule on
$\mathcal{M}$ if $AB,BA\in\mathcal{J}$ whenever $A \in \mathcal{J}$ and $B \in\mathcal{M}$ \cite[Definition 2.4.5]{LSZ}. If $\mathcal{M}$ is a finite factor, then every operator bimodule is a Calkin
operator space \cite[Lemma 2.4.6]{LSZ}.

\begin{cor}\label{cor_ext_KL}
 Let $\mathcal{M}$ be a II$_1$-factor and let $\mathcal{B}\subsetneq S(\mathcal{M})$ be a $\mathcal{M}$-bimodule.
 If $D:S(\mathcal{M})\to S(\mathcal{M})$ is a derivation and $D(S(\mathcal{M}))\subset \mathcal{B}$, then $D\equiv 0.$
\end{cor}

\section{Approximate derivative as a unique extension of the classical derivation $\frac{d}{dt}$}\label{sect1}

Let $\mathcal{A} = S[0,1]$ be the  $\ast$-algebra of all classes of Lebesgue measurable functions on
$[0,1]$ (as usual, the quotient is taken with respect to the relation \lq\lq equal almost everywhere\rq\rq), which is the Murray-von Neumann algebra associated with the finite von Neumann algebra $L_\infty [0,1]$ of all (classes of) bounded functions on $[0,1]$.
Consider the algebra $D[0,1]$ of (classes of) differentiable functions that is having almost everywhere finite derivation on
$[0,1]$.
Obviously, $D[0,1]$  is a $\ast$-subalgebra of $S[0,1]$.

We denote by $\lambda$ the Lebesgue measure on $[0,1]$. Sometimes, we denote by $[f]$  the class in $S[0,1],$ containing a measurable function $f$ on $[0,1].$ However, frequently we do not distinguish between $f$ and $[f]$.

In this section we show that the classical derivation $\frac{d}{dt}$ on the algebra $D[0,1]$ of all differentiable functions on $[0,1]$ (which is correctly defined, see Proposition \ref{eight} below) extends uniquely to  the algebra of all approximately differentiable functions  that is having almost everywhere finite approximative derivation. Furthermore, this algebra is the largest $*$-algebra in $S[0,1]$ which admits a unique extension of this derivation.

\subsection{The classical derivation $\frac{d}{dt}$ on $D[0,1]$}

 We start by showing that the classical derivation $\frac{d}{dt}$ is well-defined on the $*$-algebra $D[0,1]$.

Note that for any differentiable function $f\in S[0,1]$, the derivative $f'$  is a measurable function as the pointwise limit of a sequence on measurable functions.
\begin{prop}\label{eight}
Let  $f$ and $g$ be almost everywhere differentiable functions on $[0,1]$ and $f=g$ almost everywhere. Then the set of all points in which $f$ and  $g$ simultaneously have finite derivative has full measure, and the
derivatives $f'$ and $g'$ are measurable and equal almost everywhere.
\end{prop}

\begin{proof}
Let $A$ be the set of all points $t\in [0,1]$ such that $f(t)=g(t)$ and both derivatives $f',\,  g'$ exist and finite.
By \cite[Theorem 3.1.4]{Fed}, $f'$ and $g'$ are measurable on $A.$  The function  $h=f-g$ has everywhere defined derivative $f'-g'$ on $A.$  Since  $h(t)=0$ for all $t\in A,$  the equality $h'(t)=0$ holds on $A.$  The proof is complete since the latter set has full measure.
\end{proof}

Proposition \ref{eight} allows us to correctly define the classical derivation $\partial:D[0,1]\to S[0,1]$.
\begin{defn}\label{def_classical_der}
We define the derivation $\partial:D[0,1]\to S[0,1]$ by setting
$$\partial([f])=[f'],\quad [f]\in D[0,1].$$
\end{defn}

The following proposition establishes that the classical derivation $\partial$ on the algebra $D[0,1]$ is non-expansive (see Definition \ref{def_nonexpansive}). In particular, results of \cite{BSCh06} are applicable to $\partial$.
\begin{prop}\label{nine}
The derivation  $\partial:D[0,1]\to S[0,1]$ is non-expansive.
\end{prop}

\begin{proof}
Let  $f$ be almost everywhere differentiable on $[0,1]$
and suppose that the set  $N(f):=\{t\in [0,1]:\ f(t)=0\}$ has non-zero measure. If  $t$ is a density point of  $N(f),$ and at this point there is a derivative  of $f,$ then we have $f'(t)=0.$  Thus $N(f)$ is a subset of the set $N(f'):=\{t\in [0,1]:\ f'(t)=0\}.$  This means that $s(\partial(f))=s(f')\leq s(f).$ The proof is complete.
\end{proof}

\begin{prop}Let $E$ be the fat Cantor set in $[0,1]$ (also known, as  Smith-Volterra-Cantor set). Then the characteristic function $\chi_E\notin D[0,1]$. In particular, the algebra $D[0,1]$ does not contain all projections from $S[0,1]$.
\end{prop}
\begin{proof}
The set $E$ is a closed nowhere dense subset of $[0,1]$ with $\lambda(E)>0$. Denote by  $F$ the set of all points from $E$, which are density points for  $E$. We have that  $\mu(F)=\mu(E)>0$ (due to Lebesgue density theorem). Since the set $E$ is nowhere dense, it follows that in every neighbourhood of a point $t\in F$ there exist points, which do not belong to $E$. It means that finite derivative $(\chi_E)'(t)$  does not exist at any point $t\in F$. Consequently,  $[\chi_E]\notin D[a,b]$, as required.
\end{proof}

\subsection{The $*$-algebra $AD[0,1]$ of approximately differentiable functions}\label{sec_AD_intro}

We recall firstly the concept of approximately differentiable functions.

Consider a Lebesgue measurable set $E\subset \mathbb{R},$
a measurable function $f: E \to \mathbb{R}$ and a point $t_0\in E,$  where $E$ has Lebesgue  density equal to $1.$ If the approximate limit
$$
f'_{ap}(t_0):=\textrm{ap}-\lim\limits_{t\to t_0}\frac{f(t)-f(t_0)}{t-t_0}
$$
exists and  it is finite, then it is called approximate derivative of the function $f$ at $t_0$ and the function is called approximately differentiable at $t_0$ (see \cite{Saks} for the details).

Note that by Lebesgue density theorem,  for any measurable subset $A$ of $[0,1]$ almost every point is Lebesgue density point of $A$. Therefore, the following definition makes sense.
\begin{defn}
Let $AD[0,1]$ be the set of all classes $[f]\in S[0,1]$, for which $f$ is approximately differentiable almost everywhere.
\end{defn}

Since a density point of two subsets $E$ and $F$ is a density point of the intersection $E\cap F$, it follows that the sum and product of two approximately differentiable functions  is again approximately differentiable. Therefore,  $AD[0,1]$ is a $*$-subalgebra of $S[0,1]$.

\begin{prop}\label{ad01}
The $*$-algebra $AD[0,1]$ is a regular proper $*$-subalgebra of $S[0,1]$  containing $D[0,1]$ and all projections from $L_\infty[0,1]$.
\end{prop}

\begin{proof}
Let $f$ be a representative of  $[f]\in AD[0,1]$.  Then  $f_{ap}'$ is a measurable function and the function $g$ on $[0, 1]$ defined as
$$
\displaystyle g(t)=\left\{\begin{array}{ll}
            \frac{1}{f(t)}, & \hbox{if}\,\,\, f(t)\neq 0; \\
            0,              & \hbox{if}\,\,\, f(t)=0.
            \end{array}
            \right.
$$
is also approximately differentiable almost everywhere in   $[0,1].$  Hence,
$g\in AD[0,1]$  and $fgf=f.$ Thus, the algebra $AD[0,1]$ is regular.

Since any differentiable function is approximately differentiable, it follows that $D[0,1]\subset AD[0,1]$.

Let us show that $AD[0,1]$ contains  all projections from
$L_\infty[0,1].$ Indeed, take a measurable subset $A$ in $[0, 1].$ Consider a subset $A_0\subseteq A$ the set of all points of
density  of $A.$ By Lebesgue's density theorem we know that Lebesgue measure of the set $A\setminus A_0$ vanishes. Since the characteristic function $\chi_{A_0}$  has
an approximate derivative equal to zero almost everywhere in $A_0,$ it follows that the class containing the function
$\chi_A$ belong to $AD[0, 1].$
Hence  $AD[0,1]$ contains all projections from $L_\infty[0,1]$.

Finally, to show that $AD[0,1]$ is a proper subalgebra of $S[0,1]$, let $f$ be a measurable function $[0,1]$ which is not approximately differentiable almost everywhere on $[0,1]$ (such function exists as shown in \cite[Chap. IX, \S 11]{Saks}).
Let  $f \sim g$ and let  $g$ has an approximate derivative at point $t_0.$
Let $A\subseteq [0,1]$ be a measurable subset with $\lambda(A)>0$ such that $t_0\in A.$  Since $f \sim g,$ it follows that
the set  $A$ and its subset $A\cap \{t\in [0,1]: f(t)=g(t)\}$ have same measure   and therefore their sets  of all density points also coincide.
Therefore a function $f$ also has a finite approximate derivative at point $t_0.$ Hence, the function $g$ does not admit a finite approximate derivative
almost everywhere on $[0,1].$
Due to the arbitrary choice of $g\sim f$, we conclude that $[f] \notin AD[0,1].$ The proof is complete.
\end{proof}

We need the following characterization of the algebra $AD[0,1]$.

\begin{prop}\label{prop_AD_repr}The $*$-subalgebra $AD[0,1]$ coincides with the set of all  functions of the form
\begin{equation}\label{eq_AD_repr}
\sum\limits_{n=1}^\infty \chi_{A_n} g_n,
\end{equation}
with $ A_n \cap
A_k = \emptyset, n \neq k, \lambda(\bigcup A_n) = 1$ and $g_n\in C^1[0,1]$,  $n\in \mathbb{N}.$
\end{prop}
\begin{proof}
We prove firstly that any function of the form \eqref{eq_AD_repr}
is approximately differentiable almost everywhere.

For each $i$ denote by $\widetilde{A_i}$ the set of all density points $t$ of $A_i$ such that there is a finite derivative $g'_i(t).$ Since $g_i$  is almost everywhere differentiable, due to Lebesgue density theorem \cite[Theorem 10.2]{Saks}, we obtain that $\lambda\left(A_i  \bigtriangleup \widetilde{A_i}\right)=0.$  Then
at each point  $t\in A_i\cap \widetilde{A_i}$  a function  $f$ has an approximate derivative equal to  $g'_i(t).$
Therefore  $f\in AD[0,1]$.

The converse inclusion follows from the fact that any approximately differentiable function is continuously differentiable outside of a set of arbitrarily small measure \cite[Theorem 3.1.16]{Fed}.
\end{proof}

Recall (see Section \ref{sec_der_prel}) that the complete metric $\rho$ on $S[0,1]$ is defined by
$$\rho(x,y)=\lambda(s(x-y)),\quad x,y\in S[0,1].$$ We say that a $*$-subalgebra $\mathcal{A}\subset S[0,1]$ is \emph{topologically closed} if $(\mathcal{A}, \rho)$ is a complete metric space.

\begin{prop}\label{prop_AD_smallest_reg}
The $*$-algebra $AD[0,1]$ is the smallest regular, topologically closed  $*$-subalgebra of $S[0,1]$  containing $D[0,1]$ and all projections from $S[0,1]$.
\end{prop}
\begin{proof}

By Proposition \ref{ad01} the algebra $AD[0,1]$ is regular and contains $D[0,1]$ and all projections from $S[0,1]$. We now show that $AD[0,1]$ is topologically closed.
Let $[f]$ be a $\rho$-limit point of $AD[0,1].$
Then for each
$n\in \mathbb{N}$ there is a measurable subset $A_n$ in  $s(f)$
such that  $\lambda(s(f) \backslash  A_n)<1/n$ and $[f \chi_{A_n}] \in AD[0,1]$. Hence,
$f=\sum\limits_{n=1}^\infty f_n$, where  $f_1=f \chi_{A_1},\,   f_n=f \chi_{A_n \setminus \cup_{k=1}^{n-1} A_k}$ for $n>1.$ By Proposition \ref{prop_AD_repr}, every $f_n$ is of the form \eqref{eq_AD_repr}, and therefore,  $f$ is also of the form \eqref{eq_AD_repr}. Using again Proposition \ref{prop_AD_repr} we conclude that $[f]\in AD[0,1]$, that is the algebra $AD[0,1]$ is topologically closed.

Let $\mathcal{A}\subset AD[0,1]$ be a regular, topologically closed $*$-subalgebra of $S[0,1]$, which contains $D[0,1]$ and all projections from $S[0,1]$. Let $f\in AD[0,1]$. By Proposition \ref{prop_AD_repr}, $f$ has the form
$$f=\sum\limits_{n=1}^\infty \chi_{A_n} g_n,$$ for some
with $ A_n \cap
A_k = \emptyset, n \neq k, \lambda(\bigcup A_n) = 1$ and $g_n\in C^1[0,1]$, $n\in \mathbb{N}.$ Note that the partial sums $\sum_{n=1}^k \chi_{A_n} g_n,\ k\in\mathbb{N}$ are contained in $\mathcal{A}$. By \cite[Proposition 2.7]{BSCh06} the series $f=\sum\limits_{n=1}^\infty \chi_{A_n} g_n,$ converges with respect  the metric $\rho$. Therefore, $f\in \mathcal{A}$, that is $\mathcal{A}=AD[0,1]$.
\end{proof}

\begin{defn}Let  $\mathcal{A}$ be a $\ast$-subalgebra of $S[0,1]$. Denote by $\mathcal{A}[x]$ the $\ast$-algebra of all polynomials
with coefficients from $\mathcal{A}.$ An element $a\in S[0,1]$ is said to be
\emph{integral} with respect to $\mathcal{A},$ if there exists a unitary polynomial $p\in  \mathcal{A}[x]$ such that $p(a)=0$.
 The algebra $\mathcal{A}$  is said to be \emph{integrally closed} if it contains all elements from $S[0,1]$, which are integral with respect to $\mathcal{A}$.

\end{defn}

\begin{prop}\label{prop_AD_int_closed}
The algebra $AD[0,1]$ is integrally closed.
\end{prop}

\begin{proof}

Let us firstly consider the special case when $[f]\in S[0,1]$ is integral with respect to $D[0,1],$
i.e.,  $f$ is a root of unitary polynomial
$p(x) = x^m  + a_1 x^{m-1} + \ldots  + a_{m}$ with coefficients $a_k,$ $1\leq k\leq m$ being given by almost everywhere differentiable functions on $[0,1]$.
In addition, we can assume that for almost all   $t \in [0,1]$ the number
$f(t)$ is a simple root of a complex polynomial $p(\xi) = \xi^m  + a_1(t) \xi^{m-1} + \ldots  + a_{m}(t)$ (see \cite[Proposition 3.3]{BSCh06}).

If  $m = 1,$  then $f$ is a root of unitary polynomial
$p(x) = x + a_1,$ and therefore $f =  -a_1$  is  almost everywhere differentiable on
$[0,1].$
Therefore, we assume that $m\geq 2$.
By the choice of $p$, for almost all points  $t \in [0,1]$ the scalar
$f(t)$ is a simple root of a polynomial $p(\xi).$ Let us fix one of such points $t_0\in [0,1]$ and
set  $z_0 = (a_1(t_0), \ldots, a_{m}(t_0), f(t_0)) \in \mathbb{C}^{m+1}.$ Consider the function $F$ on $\mathbb{C}^{m+1}$ defined by
$$
F(\xi_1, \ldots, \xi_{m}, y) = y^{m} +
\xi_1 y^{m-1} + \ldots  + \xi_{m}.
$$
It is differentiable on
${\mathbb C}^{m+1},$
moreover,
\begin{center}
$F(z_0) = 0$ and $F'_y = my^{m-1} + (m-1) \xi_1 y^{m-2} + \ldots + \xi_{m-1}.$
\end{center}
Note that $F'_y(z_0) \neq 0,$ because  by our choose of $t_0,$ the number  $f(t_0)$ is a simple root of  $p(\xi).$
Since  $F'_y$ is continuous,  there is
a neighbourhood  $V(z_0)\subset\mathbb{C}^{m+1}$ of $z_0$ such that for any  $z \in V(z_0)$ we have  $F'_y(z) \neq 0.$
Moreover,
all other partial derivative is  $F'_{\xi_k} = y^{m-k}$ are continuous.
Hence, $F$ satisfy all conditions of the implicit function theorem
(see e.g.  \cite[Page 315]{LS}). Thus by implicit function theorem there exists a neighbourhood $W\subset\mathbb{C}^{m}$ of  $(a_{1}(t_0), \ldots, a_{m}(t_0))$  such that $W \subset \pi(V(z_0))$ (here a projection $\pi:\mathbb{C}^{m+1}\to \mathbb{C}^m$ defined as $\pi(\xi_1, \ldots, \xi_{m+1})=(\xi_1, \ldots, \xi_m)$) and there is a unique differentiable function
$G : W \to \mathbb{C}$ such that
\begin{center}
$G(a_1(t_0),...,a_m(t_0)) = f(t_0)$ and  $F(w, G(w)) = 0$ for all $w \in W.$
\end{center}
Take    $\varepsilon > 0$ such that
 $(a_1(t), \ldots, a_m(t))\in W$ for almost all $t\in (t_0 - \varepsilon, t_0 + \varepsilon).$ Then
 $g(t)= G(a_1(t), \ldots, a_m(t))$ is almost everywhere differentiable on $(t_0 - \varepsilon, t_0 +
\varepsilon).$
Since $F(w, G(w)) = 0$ for all $w \in W$ and  $(a_1(t), \ldots, a_m(t))\in W$ for almost all $t\in (t_0 - \varepsilon, t_0 + \varepsilon),$ it follows that
$p(g(t)=0$ for almost all $t\in (t_0 - \varepsilon, t_0 + \varepsilon).$
Thus  $\chi_Bg$ is a  root of the polynomial  $p,$ where $B=(t_0 - \varepsilon, t_0 + \varepsilon).$
Since $f$ is also root of the polynomial  $p,$ it follows that $\chi_Bf$ is  a root of the polynomial $\displaystyle \frac{p(x)-p(g(t))}{x-g(t)}$ whose degree is strictly less than  $m.$
 Hence, $$\lambda(\left\{t: f(t)=g(t)\right\}\cap (t_0 - \varepsilon, t_0 + \varepsilon))=0,$$ that is, $f(t)$ and $g(t)$ coincide almost everywhere in  $(t_0 - \varepsilon,t_0 + \varepsilon).$
This means that $f$ is an almost everywhere differentiable function.

Now we shall  consider the general case when $[f]\in S[0,1]$ is integral with respect to $AD[0,1].$ It follows from Proposition \ref{prop_AD_repr} that $f$ is a root of a unitary polynomial
$p(x) = x^m  + a_1 x^{m-1} + \ldots  + a_{m}$ with coefficients $a_i$ ($1\leq i\leq m$), where  all $a_i$ is of the form \eqref{eq_AD_repr}.
Let $a_i = \sum\limits_{n=1}^\infty \chi_{A_{i,n}} g_{i,n},$
where  $A_{i,n} \cap
A_{i,k} = \emptyset$ for  $n \neq k,$  $\lambda\left(\bigcup\limits_{n} A_{i,n}\right) = 1$ and $g_{i,n}$ is an almost everywhere differentiable function on $[0,1]$ for all $n, i\in \mathbb{N}.$
Further consider a partition of $[0,1]$ consisting from subsets of the form
$\bigcap\limits_{i=1}^m A_{i,n_i},$ $n_1, \ldots, n_m\in \mathbb{N}.$  For each  $\bigcap\limits_{i=1}^m A_{i,n_i}$ with a non zero Lebesgue measure there exists  a sequence of  disjoint intervals  $\left\{[a_j, b_j]: j\in J\right\}$  in  $[0,1]$ (depending on $\bigcap\limits_{i=1}^m A_{i,n_i}$)  such that $\lambda\left(\bigcap\limits_{i=1}^m A_{i,n_i}\triangle \bigcup\limits_{j}[a_j, b_j]\right)=0.$ Then
$\chi_{[a_j, b_j]}f(t)$ is a root of $p_{[a_j, b_j]}(x) = x^m  + \chi_{[a_j, b_j]}a_1 x^{m-1} + \ldots  + \chi_{[a_j, b_j]}a_{m},$ where  all coefficients $\chi_{[a_j, b_j]}a_i=\chi_{[a_j, b_j]}g_{i, n(j)}$ are almost everywhere differentiable on $[a_j, b_j].$ Considering instead $[0,1]$ intervals $[a_j, b_j]$ in  the previously treated special case, we obtain that  a function  $f$ coincide with an almost differentiable function on   $[a_j, b_j].$ Using again Proposition \ref{prop_AD_repr}, we conclude that  $[f]\in AD[0,1]$, that is the algebra $AD[0,1]$ is integrally closed, as required.
\end{proof}

 Propositions \ref{prop_AD_smallest_reg} and \ref{prop_AD_int_closed} imply the following
 \begin{cor}The algebra $AD[0,1]$ is the smallest regular, topologically and integrally closed $*$-subalgebra of $S[0,1]$, which contains $D[0,1]$ and all projections from $S[0,1]$.
 \end{cor}

For a $\ast$-subalgebra $\mathcal{B}$ in $S[0,1]$ denote by $M_n(\mathcal{B})$ the $\ast$-algebra of all $n\times n$ matrices over $\mathcal{B}.$

The next Proposition will be used in the following Section.

\begin{prop}
\label{t1}
Let $h=h(t)$ be a measurable function on $[0,1]$ which is nowhere approximate differentiable.
Then for any matrix $A=\left(a_{i,j}\right)_{i,j=1}^n \in M_n(AD[0,1])$ a matrix $hE-A$ is not invertible in $M_n(S[0,1]),$
where $E$ is the unit matrix in $M_n(S[0,1]).$
\end{prop}

\begin{proof}
By  \cite[Proposition 1.3.9 (ii)]{Dales} the matrix $hE-A$ is invertible if and only if its $S[0,1]$-valued determinant $\det(hE-A)$ is invertible in $S[0,1].$ Suppose that  $\det(hE-A)$ is not invertible. This means that there exists a measurable subset  $X$ in $[0,1]$ with a non zero measure such that $\chi_X\det(hE-A)=0.$
Note that $\chi_X\det(hE-A)$ is a unitary polynomial over  $AD[0,1]$ of variable  $\chi_X h.$ By Proposition~\ref{prop_AD_int_closed}, the algebra $AD[0,1]$
is integrally closed, and therefore $\chi_X h\in AD[0,1].$ Hence  $h$ is approximately differentiable on density points $X,$ which contradicts with the choice of $h.$
\end{proof}

\subsection{Approximate derivative as the largest extension of the classical derivative}

In this subsection we show that $AD[0,1]$ is the largest $*$-subalgebra of $S[0,1]$, which admits unique extension of the classical derivation $\partial:D[0,1]\to S[0,1]$ (see Definition \ref{def_classical_der}). We start with the following

\begin{prop}\label{prop_partial_AD_nonexpansive}
Let $AD[0,1]$ be the $*$-subalgebra of $S[0,1]$ of all approximately differentiable functions. There exists a unique non-expansive derivation $\partial_{AD}:AD[0,1]\to S[0,1]$, which extends the classical derivation $\partial:D[0,1]\to S[0,1]$.
\end{prop}
\begin{proof}
By Proposition \ref{nine} the derivation $\partial:D[0,1]\to S[0,1]$ is nonexpansive.
By \cite[Proposition 2.4, Proposition 2.5]{BSCh06}, there exists a unique non-expansive extension $\delta$ of the derivation $\partial$ up to the least regular subalgebra of $S[0,1]$ generated by $D[0,1]$ and all
projections from $S[0,1]$. By \cite[Proposition 2.6]{BSCh06} the derivation $\delta$ can be extended uniquely up to a derivation on the least topologically closed regular subalgebra of $S[0,1]$ containing $D[0,1]$ and all projections from $S[0,1]$. However, by Proposition \ref{prop_AD_smallest_reg} above the latter algebra coincides with $AD[0,1]$. Thus, there exists a unique extension $\partial_{AD}:AD[0,1]\to S[0,1]$ of the derivation $\partial:D[0,1]\to S[0,1]$.
\end{proof}

To prove that $AD[0,1]$ is the largest $*$-subalgebra of $S[0,1]$, which admits unique extension of $\partial:D[0,1]\to S[0,1]$ we recall the following notions.
Let, as before, $\mathcal{A}[x]$ denote the $\ast$-algebra of all polynomials
with coefficients from a subalgebra  $\mathcal{A}\subset S[0,1]$.
An element $a\in S[0,1]$ is said to be
\begin{itemize}
\item[] \textit{algebraic} with respect to $\mathcal{A},$ if there exists a non-zero polynomial $p\in  \mathcal{A}[x],$
such that $p(a) = 0;$
\item[] \textit{transcendental} with respect to $\mathcal{A},$ if $a$ is not algebraic over $\mathcal{A};$
\item[] \textit{weakly transcendental} with respect to $\mathcal{A},$ if $a\neq 0,$ and for any non-zero
idempotent $e \leq  s(a)$ the element $ea$ is not integral with respect to $\mathcal{A}.$
\end{itemize}

\begin{thm} The $*$-algebra $AD[0,1]$ is the largest subalgebra of $S[0,1]$, containing the algebra $D[0,1]$, which admits unique extension of the derivation $\partial:D[0,1]\to S[0,1]$.
\end{thm}
\begin{proof}

Suppose that $\mathcal{A}$ is a subalgebra of $S[0,1]$, such that $D[0,1]\subset AD[0,1]\subset \mathcal{A}$ and there exists a unique extension $\delta:\mathcal{A}\to S[0,1]$ of the derivation $\partial:D[0,1]\to S[0,1]$. We claim that $\mathcal{A}\subset AD[0,1]$.

Assume the contrary. Then there exists an element  $a \in \mathcal{A}$ such that $a
\notin AD[0,1]$. Let $\nabla$ be Boolean algebra of all idempotents in $S[0,1]$.
The subset  $$\nabla_a=\left\{e\in \nabla: ea \in AD[0,1]\right\}$$  is non empty, since $0\in \nabla_a.$ Let  $$e_a=\sup\limits_{e\in \nabla_a} e.$$
Take any elements
$e,\, g\in \nabla_a.$ Since  $(e\vee g)a =e a +(g-e\wedge g)a\in AD[0,1],$
it follows that $e\vee g\in \nabla_a,$ in other words the set $\nabla_a$ is closed under the operation $\vee.$
Therefore there exists an increasing net  $\{e_i\}\subset \nabla_a$ such that
$e_i \uparrow e_a.$ Thus $e_a \in \nabla_a,$ because $e_i a  \stackrel{\rho}\longrightarrow  e_a a$ and $AD[0,1]$ is $\rho$-closed (see Proposition \ref{prop_AD_smallest_reg}).

By the assumption, we have $a
\notin AD[0,1],$ and therefore $e_a\neq\mathbf{1}.$ By a construction, $e_a$ is the greatest element of $\nabla_a,$ hence  $fa \notin AD[0,1]$ for all $0\neq f\leq \mathbf{1}-e_a.$  It follows that
$a(\mathbf{1}-e_a)$ is a weakly transcendental element with respect to  $AD[0,1].$
Otherwise, there is an idempotent    $0\neq f\leq \mathbf{1}-e_a$ such that $fa$ is integral with respect to $AD[0,1]$.
But $AD[0,1]$ is integrally closed (see Proposition \ref{prop_AD_int_closed}), and hence we should have $fa\in AD[0,1],$ which is impossible due to the maximality of the element $e_a.$

Let  $\mathcal{B}$ be  a $\ast$-subalgebra generated  by $\mathcal{A}$ and
$AD[0,1].$ Since $a(\mathbf{1}-e_a)$ is a weakly transcendental element with respect to  $AD[0,1],$ \cite[Proposition~3.7]{BSCh06} implies that on $\mathcal{B}$
there exist  derivations  $\delta_1$ and  $\delta_2,$
extending  $\partial_{AD}$ such that $\delta_1(a(\mathbf{1}-e_a)) = 0$ and $\delta_2(a(\mathbf{1}-e_a)) = \mathbf{1}-e_a.$
Hence  $\delta_1(a) =
\partial_{AD}(ae_a)$ and $\delta_2(a) = \partial_{AD}(ae_a) + \mathbf{1}-e_a.$ Thus  the  restrictions  of $\delta_1$ and  $\delta_2$ onto  $\mathcal{A}$ are different extensions of $\partial$ onto $\mathcal{A},$ which contradicts the assumption  that $\partial$ admits a unique extension onto $\mathcal{A}$. This completes the proof.
\end{proof}

 \begin{rem}
 It should be pointed out that there are various extensions and generalizations of the classical derivation $\frac{d}{dt}$ as well as various classes of differentiable functions corresponding to such generalizations (see e.g. \cite{Saks}). The special interest attached to the notion of approximate differentiation and its corresponding class $AD[0,1]$ is justified by the fact that the algebra $AD[0,1]$ is the largest subalgebra of $S[0, 1]$ admitting a unique extension of the classical derivation $\frac{d}{dt}$.
 \end{rem}

\section{The algebra of approximately differentiable operators affiliated with hyperfinite type II$_1$ factor and its derivations}\label{der_on_factors}

In this section we introduce an analogue of approximately differentiable functions for hyperfinite type II$_1$ factor $\mathcal{R}$. As in the commutative case, they form a regular $*$-subalgebra $AD(\mathcal{R})$ in the algebra  $S(\mathcal{R})$. We show that there exists a derivation $\delta_{AD}:AD(\mathcal{R})\to S(\mathcal{R})$ which extends the classical approximate derivative $\partial_{AD}$, discussed in Section \ref{sect1}.

The contents of this section complement and extend results from seminal work due to von Neumann \cite{Neu}. In that work, a regular ring  $C_\infty$ of continuous geometry for $\mathbb{C}$ was constructed starting with  a sum of an increasing sequence of matrix rings over the field of complex numbers and then completed in the rank-metric. Our noncommutative analogue of approximately differentiable functions for hyperfinite type II$_1$ factor $\mathcal{R},$
is defined as a completion in the rank-metric of  a sum of an increasing sequence of matrix rings over  $AD[0,1].$

\subsection{Hyperfinite $II_1-$factor as an infinite tensor product}\label{r subsection}

The idea of approximately differentiable operators for the hyperfinite type II$_1$ factor $\mathcal{R}$ is based on the identification of $\mathcal{R}$ with the relative infinite tensor product of matrix algebras. We start with  recalling this identification and refer the reader for further details concerning this construction to
\cite{DFPS, DoSu, ScSu}.

Let $M_2(\mathbb C)$ be the algebra of all $2\times 2$ matrices over the filed $\mathbb{C}$ of all complex numbers. We denote by $\tr_2$ the normalised trace $\tr_2$ on $M_2(\mathbb C)$, that is  $\tr_2(1_2)=1$, where $1_2$ is the $2\times 2$ identity matrix.

The hyperfinite $II_1-$factor $\mathcal R$ can be identified with the relative infinite tensor product
\[\left( \mathcal R,\tau \right)\cong
	\bigotimes\limits_{k=1}^\infty
\left( M_2(\mathbb C),tr_2 \right),\]
such that $\tau$ is a faithful normal tracial state on $\mathcal R$.
The construction of relative infinite tensor products is detailed in
\cite[Definition~\text{XIV}.1.6]{Takesaki:2003}.

For each $n\ge 1$, we may consider finite truncations of this infinite tensor
product.
Set  $\mathcal R_0=\mathbb C$, and let
\[\left(\mathcal R_n,\tau_n\right)=
	\bigotimes\limits_{k=1}^n
\left( M_2\left(\mathbb C \right),tr_2 \right)=M_2(\mathbb{C})^{\otimes n},\]
be the matrix space of $2^n\times 2^n$ matrices, with the normalised trace
$\tau_n$.

There is a natural inclusion $\iota_n:\mathcal R_n\hookrightarrow \mathcal R$,
for every $n\ge 0$, given by
\[\iota_n(x)=x\otimes\left( \bigotimes_{k=n+1}^\infty1_{m(k)} \right).\]
We will identify each $\mathcal R_n$ with its image $\iota_n(\mathcal R_n )\subset\mathcal R$.
Each $\mathcal R_n$ is trivially a von Neumann subalgebra of $\mathcal R$,
and the restriction of $\tau$ to $\mathcal R_n$ gives the trace $\tau_n$.

Thus, the spaces $(\mathcal R_n)_{n=0}^\infty$ form an increasing
filtration of $\mathcal R$, and by definition of the infinite tensor product,
the union $\bigcup_{n=0}^\infty \mathcal R_n$ is weak-$^*$ dense in
$\mathcal R$.

Alternatively, it will also be useful to consider the spaces $\mathcal R_n$
as vector valued matrix spaces.
In particular, for each $n\ge 1$, the space $\mathcal R_n$
is isomorphic to $M_2\left(\mathcal R_{n-1} \right)$, the space
of $\mathcal{R}_{n-1}$-valued $2\times 2$ matrices.
This follows as
$\mathcal R_n=\mathcal R_{n-1}\otimes M_2\left(\mathbb C \right)$,
by definition.

\subsection{$*$-algebra of approximately differentiable operators}\label{nc_approx_differentiability}

Let $D_2(\mathbb{C})$ be the diagonal subalgebra in $M_2(\mathbb{C})$ and consider the maximal abelian subalgebra $\mathcal{D}$ in $\mathcal{R}$, defined by
$$\mathcal{D}=\bigotimes_{k=1}^{\infty}\left(D_2(\mathbb{C}),tr_2 \right).$$
It is known \cite[Theorem 3.2]{Tau} that $\mathcal{D}$ is a Cartan subalgebra in $\mathcal{R}.$

We identify
$$(\mathcal{D},\tau)=\bigotimes_{k=1}^{\infty}L_{\infty}(\{0,1\},\nu),$$
where $\nu(0)=\nu(1)=\frac12.$ The latter algebra is identified with $L_{\infty}(0,1)$ equipped with the usual Lebesgue integration.

We specify the $*$-isomorphism $\pi$ of the algebras $L_{\infty}[0,1]$ and $\mathcal{D}.$

 Consider subsets $\displaystyle X_k=\bigcup\limits_{l=0}^{2^{k-1}-1}\left[\frac{2l}{2^k}, \frac{2l+1}{2^k}\right),\ k\in\mathbb{N}.$
Define the mapping
\begin{eqnarray}\label{isoo}
\pi(\chi_{X_{k}})=\left(\bigotimes_{i=1}^{k-1} 1_i\right)\otimes\begin{pmatrix}1&0\\0&0\end{pmatrix}\otimes \left(\bigotimes_{i=k+1}^\infty 1_i\right),
\end{eqnarray}
where $1_i$ is the $2\times 2$ identity matrix.
The system $\{2\chi_{_{X_k}}-1: k\in \mathbb{N}\}$ is the Rademacher system of functions on $[0,1]$, that is a system of independent random variables taking values $1$ and $-1$ with probability $1/2$. The span of  $\{1, \chi_{_{X_k}}: k\in \mathbb{N}\}$ is dense in $L_\infty[0,1]$ in measure. Therefore, the mapping $\pi$
uniquely extends to a $\ast$-isomorphism $\pi: L_\infty[0,1]\longrightarrow\mathcal{D}.$
Since the $*$-isomorphism $\pi$ preserves the trace, it follows from Proposition \ref{prop_trace_pres_isom} that it  extends up to $*$-isomorphism $\pi: S[0,1]\to S(\mathcal{D}).$

 Therefore, throughout this section we do not distinguish between the $*$-algebra $\mathcal{D}$ (respectively, $S(\mathcal{D})$) and the $*$-algebra $L_\infty[0,1]$ (respectively, $S[0,1]$).
 In particular, we identify $S(0,1)$  with a $*$-subalgebra of $S(\mathcal{R})$.

For $n\geq1,$ the algebra $\mathcal{R}_n=M_2(\mathbb{C})^{\otimes n}$ is spanned by the \lq\lq matrix units\rq\rq $e_{{\bf i},{\bf j}}.$ Here, ${\bf i}=(i_k)_{k=0}^{n-1}\in\{0,1\}^n,$ ${\bf j}=(j_k)_{k=0}^{n-1}\in\{0,1\}^n,$
and
$$e_{{\bf i},{\bf j}}=\bigotimes_{k=0}^{n-1}e_{i_k,j_k}.$$
%Once again, we emphasize that we identify elements from $M_2(\mathbb{C})^{\otimes n}$ with elements from $\mathcal{R}_n$ using the embedding $\iota_n$ (see the preceding section).
Therefore, each matrix $x \in \mathcal{R}_n$ has the form
$$
x=\sum_{\bf{i},\bf{j}}a_{\bf{i},\bf{j}}e_{{\bf i},{\bf j}}, \quad a_{\bf{i},\bf{j}}\in \mathbb{C}.
$$

For each pair induces ${\bf i}=(i_k)_{k=0}^{n-1}\in\{0,1\}^n,$ ${\bf j}=(j_k)_{k=0}^{n-1}\in\{0,1\}^n$  consider  a mapping
from $S(\mathcal{D})e_{{\bf i},{\bf i}}$ onto $S(\mathcal{D})e_{{\bf j},{\bf j}}$ defined as follows
$$
\pi(x)e_{{\bf i},{\bf i}} \hookrightarrow e_{{\bf j},{\bf i}} \pi(x) e_{{\bf i},{\bf j}},\,\, \pi(x)\in S(\mathcal{D}).
$$
It induces a mapping from $S[0,1]\chi_{\left[\frac{i}{2^n}, \frac{i+1}{2^n}\right]}$ onto
 $S[0,1]\chi_{\left[\frac{j}{2^n}, \frac{j+1}{2^n}\right]}$ defined as
$$
x(t)\chi_{\left[\frac{i}{2^n}, \frac{i+1}{2^n}\right]}\longrightarrow x\left(t+\frac{j-i}{2^n}\right)\chi_{\left[\frac{j}{2^n}, \frac{j+1}{2^n}\right]},
$$
where $i=\sum\limits_{k=0}^{n-1}i_k2^{n-k-1},$ $j=\sum\limits_{k=0}^{n-1}j_k2^{n-k-1}.$

Let $AD(\mathcal{D})$ be the (commutative) $*$-algebra of all approximately differentiable functions on $(0,1)$ (see Section~\ref{sect1}), which we identify with a subspace of $S(\mathcal{R})$.

Recall that the complete metric $\rho$ on $S(\mathcal{R})$ is defined by setting
$$\rho(x,y)=\tau(l(x-y))=\tau(r(x-y)),\quad x,y\in S({ \mathcal{R}}).$$
As we mentioned above in Section~\ref{sec_der_prel} the rank-metric $\rho$ was firstly introduced by J. von Neumann in \cite{Neu37} (see also \cite[pp. 160-161]{Neu62}).

We now introduce a noncommutative analogue of the algebra $AD[0,1]$, discussed at length in Section~\ref{sect1}.

\begin{defn}\label{def_AD_R}
Let $n\geq 0$ and let $\mathcal{A}_n:=AD(\mathcal{R}_n)$ be the $*$-subalgebra of $S(\mathcal{R})$ generated by $\mathcal{R}_n$ and $AD(\mathcal{D})$. We define the $*$-algebra $AD(\mathcal{R})$ of approximately differentiable operators in $S(\mathcal{R})$ by setting
$$
AD(\mathcal{R})=\overline{\bigcup_{n\geq0}\mathcal{A}_n}^{\rho}.
$$
\end{defn}

It is important to emphasize the connection of Definition \ref{def_AD_R} with seminal von Neumann paper \cite{Neu}.  Indeed, the algebra $AD(\mathcal{R})$ contains a regular ring of continuous geometry for $\mathbb{C},$ introduced in \cite{Neu}. Recall that
$$\mathcal{R}_\infty=\overline{\bigcup_{n\geq0}\mathcal{R}_n}^{\rho}$$ is a continuous geometry for $\mathbb{C}$ (see \cite[Theorems~D~and~E]{Neu}), and contained in
$AD(\mathcal{R}).$  Below, in Proposition \ref{p4} we shall prove that the algebra $AD(\mathcal{R})$ is a proper $*$-subalgebra in $S(\mathcal{R})$ and this may be seen as an extension of von Neumann results \cite[Theorem E]{Neu}.

To establish the regularity of the algebra $AD(\mathcal{R})$ and to introduce a derivation on $AD(\mathcal{R})$, we prove firstly the following auxiliary result for the $*$-algebra $AD(\mathcal{R}_n)$.

\begin{lm}\label{lem_AD(R_n)}
The $*$-algebra $AD(\mathcal{R}_n)$ is regular and every
$x\in AD(\mathcal{R}_n)$ can be written, not necessarily uniquely, as
\begin{equation}\label{an linear span eq}
x=\sum_{U\in\Pi_n}x_UU,\quad x_U\in AD(\mathcal{D}).
\end{equation}
Here $\Pi_n$ is the collection of all permutation matrices from $\mathcal{R}_n.$
\end{lm}
\begin{proof}Note that the algebra $AD(\mathcal{R}_n)$ is generated by $AD(\mathcal{D})$ and $\Pi_n.$ That is, any  $x\in AD(\mathcal{R}_n)$ can be written as a linear span of monomials of the form
$$a_1U_1a_2U_2\cdots a_mU_m,$$ for some $m\in \mathbb{N}$ and $a_k\in AD(\mathcal{D})$ and $U_k\in \Pi_n$ for every $1\leq k\leq m.$ Note that
$$UaU^{-1}\in AD(\mathcal{D})$$
for every $a\in AD(\mathcal{D})$ and for every permutation matrix $U.$ Since $\prod\limits_{l=1}^{k-1}U_l$ is again a permutation matrix, it follows that there exists $b_k\in AD(\mathcal{D})$ such that
\begin{equation}
\label{ad permutation eq}
\prod_{l=1}^{k-1}U_l\cdot a_k=b_k\cdot\prod_{l=1}^{k-1}U_l.
\end{equation}

In particular, $AD(\mathcal{R}_n)$ is a matrix ring over $AD(\mathcal{D})$.  By Proposition \ref{prop_AD_smallest_reg} the algebra $AD(\mathcal{D})=AD[0,1]$ is regular. Since every matrix ring over a regular ring is also regular
(see e.g. \cite[Theorem  3]{Skor}), it follows that $AD(\mathcal{R}_n)$ is regular.

By repeated application of \eqref{ad permutation eq}, we have that
$$a_1U_1a_2U_2\cdots a_mU_m=a_1b_2\cdots b_m\cdot U_1\cdots U_m.$$
Hence, any $x\in AD(\mathcal{R}_n)$ can be represented as in \eqref{an linear span eq}.
\end{proof}

\begin{prop}
\label{p4} The subalgebra $AD(\mathcal{R})$ is a proper regular $*$-subalgebra in $S(\mathcal{R})$ which is dense in $S(\mathcal{R})$ in the measure topology.
\end{prop}
\begin{proof}
By Lemma \ref{lem_AD(R_n)} the algebra $AD(\mathcal{R}_n)$ is a regular $*$-algebra. Since $\{AD(\mathcal{R}_n)\}_{n=0}^\infty$  is an increasing  sequence of subalgebras in $S(\mathcal{R}),$ it follows that $\mathcal{A}:=\bigcup_{n=0}^\infty AD(\mathcal{R}_n)$ is also a regular subalgebra in $S(\mathcal{R})$. Hence, $AD(\mathcal{R})$ is also regular as the closure of a regular algebra with respect to the metric $\rho$ (see e.g. \cite[Theorem 19.6]{Goodearl}).

Let $x\in \mathcal{R}$ and $\|x\|\leq 1.$
Since the $*$-subalgebra $\bigcup_{n=1}^\infty \mathcal{R}_n\subset AD(\mathcal{R})$  is  dense in $\mathcal{R}$ in the strong operator topology,
there exists a net $\{x_\alpha\}$ from the unit ball in $\bigcup_{n=1}^\infty \mathcal{M}_n$
such that
$x_\alpha \stackrel{so}\longrightarrow x.$
Then $\tau((x_\alpha-x)^\ast (x_\alpha-x))\to 0$ (see \cite[Page 130]{SZ}). This means that the  net $\{x_\alpha\}$ converges to $x$
in the norm $\|\cdot\|_2,$ where $\|z\|_2=\sqrt{\tau(z^\ast z)},\, z\in \mathcal{R}.$ Since convergence in the norm $\|\cdot\|_2$
implies convergence in measure {topology} (see \cite[Theorem 5]{Nel}), the net $\{x_\alpha\}$ converges to $x$ in the measure topology. So,
 $\bigcup_{n=1}^\infty \mathcal{R}_n$ is dense in the measure topology in $\mathcal{R},$ and is, therefore, dense in $S(\mathcal{R}).$

 Now we  show that $AD(\mathcal{R})$ is a proper subalgebra of $S(\mathcal{R})$.
For every $n\in \mathbb{N}$ take  a continuous piecewise-linear function $h_n$ on $[0,1]$ defined as follows
$$
\displaystyle
h_n(t)=\left\{\begin{array}{lll}
0, & \hbox{if}\,\,\, t=\frac{2l}{2^{16^n}},\,\,  l=0, 1, 2, \ldots, 2^{16^n-1}; \\
\frac{1}{2^{16^n-4^{n+1}}}, & \hbox{if}\,\,\, t=\frac{2l+1}{2^{16^n}},\,\,  l=0, 1, 2, \ldots, 2^{16^n-1}-1; \\
\hbox{linear}, & \hbox{if}\,\,\, \frac{l}{2^{16^n}} \leq t \leq \frac{l+1}{2^{16^n}},\,\, l=0, 1, 2, \ldots, 2^{16^n}-1. \\
            \end{array}
            \right.
$$
This function coincides with that defined in \cite{Jarnik} for the  sequences $\{k_n=2^{4^{n+1}}\}$ and $\{d_n=2^{-16^n}\}$ (see  \cite[p. 6]{Jarnik}).
Note that $h_n(\cdot)$ is differentiable on  all of points $[0,1]$ excepting the finite number of points  $t=\frac{l}{2^{16^n}},\,\,  l=0, 1, 2, \ldots, 2^{16^n}.$
Setting
\begin{eqnarray*}
h(t)=\sum\limits_{k=0}^\infty h_k(t),\,\,\, t\in
[0,1],
\end{eqnarray*}
where the series is, in fact, uniformly convergent due to \cite[p. 7]{Jarnik}. In particular, the function $h$ is continuous but nowhere approximately differentiable (see
\cite[Theorem 1]{Jarnik}).

Let $n\in \mathbb{N}$ and  $t\in [0, 2^{-16^n+1}].$ For every $l\in \{0,
\ldots, 2^{16^n-1}-1\},$ by the definition of $h_k(t)$ it follows  that
\begin{eqnarray*}
h_k\left(t+\frac{2l}{2^{16^n}}\right)=h_k\left(t\right),\quad t\in [0,1]
\end{eqnarray*}
for $k\geq n,$ and
therefore the difference
$$h\left(\cdot+\frac{2l}{2^{16^n}}\right)-h\left(\cdot\right)\large |_{ [0, 2^{-16^n+1}]}$$ is a finite sum of the
almost everywhere approximately differentiable functions on $[0, 2^{-16^n+1}].$ So,
\begin{eqnarray}\label{differ}
h\left(\cdot+\frac{2l}{2^{16^n}}\right)-h\left(\cdot\right)\large |_{ [0, 2^{-16^n+1}]}\in AD[0, 2^{-16^n+1}]
\end{eqnarray}
for all $n\in \mathbb{N}$ and for all $l\in \{0,\ldots, 2^{16^n-1}-1\}.$

Set $\overline{h}=\pi(h),$ where   $\pi: L_\infty[0,1]\longrightarrow\mathcal{D}$ is a $\ast$-isomorphism defined prior to the statement of the proposition.

Let us show that $\rho(\overline{h}, a)=1$ for all  $a\in AD(\mathcal{R}_m),$ where $m=16^n-1, \ n\ge 1.$ Fix $a\in AD(\mathcal{R}_m).$

Consider  a system of matrix units  $\{e_{{\bf i, j}}\}_{{\bf i}\in {\bf I}}$ in $\mathcal{R}_m,$  here
${\bf I}=\left\{{\bf i}: {\bf i}=(i_k)_{k=0}^{m-1}\in\{0,1\}^m\right\}.$
Let $\mathcal{B}_m$ be a $\ast$-algebra in $S(\mathcal{R})$ generated by $S(\mathcal{D})$ and $\mathcal{R}_m.$
We shall identify $\mathcal{B}_m$ with the matrix $\ast$-algebra $M_{2^m}(S(\mathcal{D})e_{{\bf 0, 0}})$ via $\ast$-isomorphism
$$
\Psi_m: x\rightarrow \left(e_{{\bf 0, i}} x e_{{\bf j, 0}}\right)_{{\bf i,j}\in {\bf I}},
$$
where ${\bf 0}=(0,\ldots, 0).$ Observe that $\Psi_m(a)\in M_{2^m}(AD(\mathcal{D})e_{{\bf 0, 0}}).$

Note that  the $\ast$-isomorphism $\pi$ sends a function
$h\left(\cdot+\frac{2l}{2^{16^n}}\right)\chi_{\left[0, \frac{2l}{2^{16^n}}\right)}(\cdot)\in S[0, 1]$ to the element of the form
$e_{\bf 0,i}\overline{h}e_{\bf i,0}\in S(\mathcal{D})e_{\bf 0,0},$ where $2l=\sum\limits_{k=0}^{m-1}i_k2^{m-k-1}$ for ${\bf i}=(i_k)_{k=0}^{m-1}\in\{0,1\}^m.$  Combining this observation with \eqref{differ}, we arrive at
$$e_{\bf 0,i}\overline{h}e_{\bf i,0}=\overline{h}e_{\bf 0,0}+r_{\bf i},$$ where $r_{\bf i}\in AD(\mathcal{D})e_{\bf 0,0}$ for all ${\bf i}\in {\bf I}.$
Recalling that $h$ is nowhere approximately differentiable and that the element $\Psi_m(a)$  belongs $M_{2^m}(AD(\mathcal{D})e_{{\bf 0, 0}}),$ and appealing to Proposition~\ref{t1}, we infer that the matrix
$\Psi_m(a-\overline{h}-r)$ is invertible in $M_{2^m}(S(\mathcal{D})e_{{\bf 0, 0}}),$
where $r=\sum_{{\bf i}\in {\bf I}} r_{\bf i}e_{\bf i,i}.$ In other words, the support projection of the element $a-\overline{h}-r$ is the identity of the algebra $\mathcal{R}$.
Hence, $\rho\left(a-r, \overline{h}\right)=1.$  Observing that the definition of the element $r\in AD(\mathcal{D}),$ does not depend on the choice of $a$, we may
replace $a$ with $a+r$ and obtain that
$\rho\left(a, \overline{h}\right)=1.$
Thus
$\rho(a, \overline{h})=1$ for all $a\in \bigcup_{m\geq0}AD(\mathcal{R}_m),$ hence   $\overline{h}\in S(\mathcal{R})\setminus AD(\mathcal{R}).$
 The proof is completed. \end{proof}

\subsection{Approximate derivation on the algebra $AD(\mathcal{R})$}

We now construct the derivation $\delta_{AD}:AD(\mathcal{R})\to S(\mathcal{R})$, which extends the approximate derivative $\partial_{AD}$ on the $*$-algebra $AD(\mathcal{D})=AD[0,1]$, introduced in Section \ref{sect1}. We start by constructing a tower of derivations on the $*$-algebras $AD(\mathcal{R}_n)$, $n\ge 1$.

Recall that $\Pi_n$ is the collection of all permutation matrices from $\mathcal{R}_n.$
By Lemma~\ref{lem_AD(R_n)} every element $x\in AD(\mathcal{R}_n)$ can be represented as
$x=\sum_{U\in\Pi_n}x_UU$ for some $x_U\in AD(\mathcal{D})$.

We define
\begin{equation}\label{def_delta_AD_Rn}
\delta_n(x)=\sum_{U\in\Pi_n}\partial_{AD}(x_U)U, \quad n\geq 1.
\end{equation}
For convenience, we also denote $\delta_0=\partial_{AD}$.

\begin{lm}The mapping $\delta_n: AD(\mathcal{R}_n)\to S(\mathcal{R})$, $n\geq 1$, given by \eqref{def_delta_AD_Rn}, is a well-defined linear mapping.
\end{lm}

\begin{proof}
Let $n\geq 1$ be fixed.
It is sufficient to show that if
$\sum_{k=1}^mx_kA_k=0$ for some $x_k\in AD(\mathcal{D})$ and $A_k\in \mathcal{R}_n$, $k=1,\dots, m, m\in\mathbb{N}$, then $\sum_{k=1}^m\partial_{AD}(x_k)A_k=0.$

 Recall that for $n\geq1,$ the algebra $\mathcal{R}_n=M_2(\mathbb{C})^{\otimes n}$ is spanned by the \lq\lq matrix units\rq\rq $e_{{\bf i},{\bf j}},$ where, ${\bf i}=(i_k)_{k=0}^{n-1}\in\{0,1\}^n,$ ${\bf j}=(j_k)_{k=0}^{n-1}\in\{0,1\}^n,$
and $e_{{\bf i},{\bf j}}=\bigotimes_{k=0}^{n-1}e_{i_k,j_k}.$
%Once again, we emphasize that we identify elements from $M_2(\mathbb{C})^{\otimes n}$ with elements from $\mathcal{R}_n$ using the embedding $\iota_n$ (see the preceding section).
Therefore, each matrix $A_k\in \mathcal{R}_n$ has the form.
$$
A_k=\sum_{\bf{i},\bf{j}}a^k_{\bf{i},\bf{j}}e_{{\bf i},{\bf j}}, \quad a^k_{\bf{i},\bf{j}}\in \mathbb{C}.
$$
For each pair of fixed indices ${\bf i_0},{\bf j_0}$ multiplying the equality
$\sum\limits_{k=1}^mx_kA_k=0$ by  $e_{{\bf i_0},{\bf i_0}}$ on the left side and by $e_{{\bf j_0},{\bf i_0}}$ on the right, we obtain that $$\sum_{k=1}^m a^k_{{\bf i_0},{\bf j_0}}x_k e_{{\bf i_0},{\bf i_0}}=0.$$

Note that $e_{{\bf i_0},{\bf i_0}}\in AD(\mathcal{D}),$ and  $e_{{\bf i_0},{\bf i_0}}$ is a projection. Therefore, since the derivation $\partial_{AD}$ vanishes on projections, it follows that from the Leibniz rule that
%
%Since  $\{x_k\}_{k=1}^m,$ $e_{{\bf i_0},{\bf i_0}}\in AD(\mathcal{D}),$ and  $e_{{\bf i_0},{\bf i_0}}$ is a projection and since  the derivation $\delta_0$ vanishes on every projection, it follows that
$$
\sum_{k=1}^m a^k_{{\bf i_0},{\bf j_0}}\partial_{AD}(x_k) e_{{\bf i_0},{\bf i_0}}=\partial_{AD}\left(\sum_{k=1}^m a^k_{{\bf i_0},{\bf j_0}}x_k e_{{\bf i_0},{\bf i_0}}\right)=0.
$$
Multiplying {the} last equality by $e_{{\bf i_0},{\bf j_0}}$ from  the right, we obtain
$$
\sum_{k=1}^m a^k_{{\bf i_0},{\bf j_0}}\partial_{AD}(x_k) e_{{\bf i_0},{\bf j_0}}=0.
$$
Hence,
$$
\sum_{k=1}^m\partial_{AD}(x_k)A_k=\sum_{k=1}^m\sum_{{\bf i},{\bf j}}a^k_{{\bf i},{\bf j}}\partial_{AD}(x_k) e_{{\bf i},{\bf j}}=\sum_{{\bf i},{\bf j}}\sum_{k=1}^m a^k_{{\bf i},{\bf j}}\partial_{AD}(x_k) e_{{\bf i},{\bf j}}=0,
$$
which completes the proof.
\end{proof}

Recall that we identify the corresponding elements from $S(\mathcal{D})$ and from $S(0,1)$.

\begin{prop}\label{permutation}
Let $n\ge 1$ and let $U\in \Pi_n$ be a permutation matrix.
Then  there exists a  permutation $\gamma:[0,1]\to [0,1]$ of dyadic intervals, such that
$$
UcU^{-1}=c\circ \gamma, \quad \forall c\in S[0,1].
$$
Moreover, if, in addition, $c\in AD[0,1]$, then $
UcU^{-1}\in AD[0,1]
$
and
$\partial_{AD}(c\circ \gamma)=\partial_{AD}(c)\circ\gamma$.
\end{prop}

\begin{proof} Let ${\bf I}=\left\{{\bf i}: {\bf i}=(i_k)_{k=0}^{n-1}\in\{0,1\}^n\right\}.$
For ${\bf i} \in {\bf I}$ set
\begin{eqnarray*}
X_{\bf i}=\bigcap_{s=0}^{n-1} Y_{s+1},
\end{eqnarray*}
where
$Y_{s+1} =X_{s+1}$ for $i_s = 0$ and $Y_{s+1}=[0, 1) \setminus X_{s+1}$ for $i_s=1.$ Here $\left\{X_k: k\in \mathbb{N}\right\}$ is the system of subsets in $[0,1]$ defined before  \eqref{isoo}.
Then $\left\{X_{{\bf i}}: {\bf i} \in {\bf I}\right\}$ is a partition of $[0,1]$ into dyadic intervals of the lengths $1/2^n.$
Using  \eqref{isoo} we obtain that
\begin{eqnarray}\label{permutat}
\pi(\chi_{X_{{\bf i}}})&=& \prod_{s=0}^{n-1} \pi(\chi_{Y_{s+1}})=\bigotimes_{s=0}^{n-1} e_{i_s, i_s}=e_{{\bf i,i}}
\end{eqnarray}
for all ${\bf i} \in {\bf I}.$

Since any permutation matrix $U\in \Pi_n$ induces a permutation of the system
$\{e_{{\bf i, i}}: {\bf i}\in {\bf I}\},$ we have that $Ue_{{\bf i, i}}U^{-1}=e_{{\bf \sigma(i), \sigma(i)}},$ where $\sigma$ is a permutation of ${\bf I}.$
Then \eqref{permutat}  gives us
$U\pi(\chi_{{\bf i}})U^{-1}=\pi(\chi_{{\bf \sigma(i)}}).$ Thus $\sigma$ induces a mapping  $\gamma:[0,1]\to [0,1]$ which acts as a permutation of dyadic intervals and  $U\pi(\chi_{X_{{\bf i}}})U^{-1}=\pi(\chi_{X_{\bf i}})\circ \gamma$ for all ${\bf i}\in {\bf I}.$
Further for $k>n,$ the subset $X_k$ the collection of disjoint dyadic intervals with the lengths $1/2^k,$ and therefore $\gamma$ permutes the elements of $X_k$. Hence,
 $U\pi(\chi_{X_k})U^{-1}=\pi(\chi_{X_k})=\pi(\chi_{X_k})\circ \gamma$ for all $k>n.$ Since a linear span of the system $\left\{1, \chi_{X_{{\bf i}}}: {\bf i} \in {\bf I}\right\}\cup
 \left\{\chi_{X_k}: k>n\right\}$ is dense in the measure topology in $S[0,1],$ it follows that
 $UcU^{-1}=c\circ \gamma$ for all $c\in S[0,1].$

Let $c\in AD[0,1].$ In order to prove  that $UcU^{-1}\in AD[0,1]$
it suffices to show
$
c \circ \gamma \in AD[0,1],$ where  $\gamma(t)=(t+r)({\rm mod}\ 1),$
$r\in (0,1)$ is a dyadic rational.
By Proposition \ref{prop_AD_repr}
it suffices to consider the case $c\in D[0,1].$
Let $c$  has a finite derivative at each point of a subset  $A\subset [0,1]$ with complete measure.
Then a subset  $\{A+r\}$ is also has complete measure, and therefore the intersection
$A\cap \{A+r\}$ also has a complete measure. For every point
$t$ of this intersection there exist finite derivatives
$c'(t)$ and  $c'(\{t+r\})$. This means that $c\circ \gamma \in D[0,1]$.

Finally,
the equality $\partial_{AD}(c(\{t+r\}))=(\partial_{AD}(c))(\{t+r\})$ implies that  $\partial_{AD}(c\circ \gamma)=\partial_{AD}(c)\circ\gamma.$
The proof is complete.
\end{proof}

Next, we show that the sequence $\{\delta_n\}_{n\ge 1}$, defined by \eqref{def_delta_AD_Rn} on the increasing sequence of algebras $AD(\mathcal{R}_n)$, is a sequence of derivations such that each of the subsequent derivation is an extension  of the previous one and all of them vanish on $\cup_{n\ge 1}\mathcal{R}_n$.

\begin{prop}\label{tower lemma}
 Let $\delta_n:AD(\mathcal{R}_n)\to S(\mathcal{R})$ be the mapping, defined by \eqref{def_delta_AD_Rn}.
For every $n\geq 0$, $\delta_n$ is a derivation and
$$\delta_{n+1}|_{AD(\mathcal{R}_n)}=\delta_n,\quad \delta_n|_{\mathcal R_n}=0.$$
In particular, $\delta_n|_{AD(\mathcal{D})}=\partial_{AD}$, $n\geq 1$.
\end{prop}

\begin{proof}
We show firstly that $\delta_n$ is a derivation for every $n\ge 1$.
By the definition of $\delta_n$, it  suffices to verify the Leibniz rule for $x=aU$ and $y=bV$ with $a,b\in AD(\mathcal{D})$ and $U,V\in \Pi_n$.

Let $U\in \Pi_n$ be fixed permutation matrix.
By Proposition \ref{permutation}  there exists a suitable permutation $\gamma:[0,1]\to [0,1]$ of dyadic intervals, such that
$$UcU^{-1}=c\circ \gamma, \quad \forall c\in S[0,1].$$
In particular,
$$UcU^{-1}\in AD[0,1], \quad \forall c\in AD[0,1].$$
Given that the permutation $\gamma$ commutes with the approximate derivative $\partial_{AD}$ we obtain that
\begin{equation}\label{u eq}
\partial_{AD}\left(UcU^{-1}\right)=\partial_{AD}(c\circ \gamma)=\partial_{AD}(c)\circ\gamma=U\partial_{AD}(c)U^{-1},\quad \forall c\in AD[0,1].
\end{equation}
Therefore, we have
$$
\delta_n(xy)=\delta_n\left(a\cdot UbU^{-1}\cdot UV\right)=\partial_{AD}\left(a\cdot UbU^{-1}\right)\cdot UV.
$$
Using the Leibniz rule for the elements $a, UbU^{-1}\in AD[0,1]$ and \eqref{u eq} we infer that
$$
\partial_{AD}\left(a\cdot UbU^{-1}\right)=\partial_{AD}(a)\cdot UbU^{-1}+a\cdot\partial_{AD}\left(UbU^{-1}\right)=\partial_{AD}(a)\cdot UbU^{-1}+aU\cdot\partial_{AD}(b)U^{-1}.$$
Hence,
\begin{eqnarray*}
\delta_n(xy) & = & \Big(\partial_{AD}(a)\cdot UbU^{-1}+aU\cdot\partial_{AD}(b)U^{-1}\Big)UV=
\\&
=& \partial_{AD}(a)U\cdot bV+aU\cdot\partial_{AD}(b)V=\delta_n(x)y+x\delta_n(y).
\end{eqnarray*}
Thus, the Leibniz rule is satisfied for $x,y$, and therefore, $\delta_n$ is a derivation on $AD(\mathcal{R}_n)$.

Since for any $x\in \mathcal{R}_n$ in the representation $x=\sum_{U\in \Pi_n}x_UU$ every $x_U$ is a constant, it follows immediately from the definition of $\delta_n$, that $\delta_n|_{\mathcal{R}_n}=0$. The equality  $\delta_n|_{AD(\mathcal{D})}=\partial_{AD}$ also follows directly, because for $x\in AD(\mathcal{D})$ the representation of $x$ in the form $x=\sum_{U\in\Pi_n}x_UU$ involves only the identical permutation matrix and so the required equality follows immediately from \eqref{def_delta_AD_Rn}.
 It remains to show that  $\delta_{n+1}|_{AD(\mathcal{R}_n)}=\delta_n,\quad n\geq 0.$

Define a derivation $\delta_n':AD(\mathcal{R}_n)\to S(\mathcal{R})$ by setting
$$\delta_n'=\delta_n-\delta_{n+1}|_{AD(\mathcal{R}_n)}.$$
We have
$$\delta_n'|_{AD(\mathcal{D})}=0,\quad \delta_n'|_{\mathcal R_n}=0.$$
Since $AD(\mathcal{R}_n)$ is generated by $AD(\mathcal{D})$ and $\mathcal R_n$ it follows from the Leibniz rule that the derivation $\delta_n'$ vanishes on $AD(\mathcal{R}_n)$. This proves the claim.
\end{proof}

We are now in a position to construct a noncommutative analogue of the approximate derivative $\partial_{AD}$.

\begin{thm}\label{b main thm} There exists a derivation $\delta_{AD}:AD(\mathcal{R})\to S(\mathcal{R})$ such that $\delta|_{AD(\mathcal{D})}=\partial_{AD}$
%and $\delta_{AD(\mathcal{R}_n)}=\delta_n$.
\end{thm}

\begin{proof}
Let $\delta_n:AD(\mathcal{R}_n)\to S(\mathcal{R})$, $n\geq 0$, be the derivation, given by \eqref{def_delta_AD_Rn}.

Consider the $\ast$-subalgebra $\mathcal{A}=\bigcup_{n\geq0}AD(\mathcal{R}_n).$ Define the mapping $\delta:\mathcal{A}\to S(\mathcal{R})$ by setting $\delta|_{AD(\mathcal{R}_n)}=\delta_n.$ By Proposition \ref{tower lemma} we have  $\delta_{n+1}|_{AD(\mathcal{R}_n)}=\delta_n$, and therefore, $\delta$ is a well-defined mapping. It is clear that $\delta$ is a derivation.

By Lemma \ref{delta continuity lemma} the derivation $\delta $ is continuous with respect to the metric $\rho$. By definition, $AD(\mathcal{R})=\overline{\bigcup_{n\geq0}AD(\mathcal{R}_n)}^{\rho}$, and so $\delta$ extends up to a derivation $\delta_{AD}:AD(\mathcal{R})\to S(\mathcal{R}).$ Since $\delta_n|_{AD(\mathcal{D})}=\partial_{AD}$, it follows that $\delta_{AD}|_{AD(\mathcal{D})}=\partial_{AD}$, which completes the proof.
\end{proof}

Next, we show  that the noncommutative approximate derivative $\delta_{AD}(\mathcal{R})$ is not spatial.

\begin{prop}\label{noninner}
\label{p3} Let $\delta_{AD}:AD(\mathcal{R})\to S(\mathcal{R})$ be as in Theorem \ref{b main thm}.
There is no $a\in S(\mathcal{R})$ such that $\delta_{AD}(x)=[a,x]$ for all  $x\in AD(\mathcal{R})$.
\end{prop}

\begin{proof}
Suppose, by contradiction,  that there exists  an element $a\in S(\mathcal{R})$ such that $\delta_{AD}(x)=[a,x]$ for all  $x\in\mathcal{B}.$ Since $S(\mathcal{R})$ equipped with the measure topology is a topological $*$-algebra it follows that,  in particular,
 $\delta_{AD}$ is continuous with respect to the measure topology.  Since
$\delta_{AD}|_{AD(\mathcal{D})}=\partial_{AD},$    it follows that $\partial_{AD}$ is also continuous in the measure topology. Furthermore, the algebra $AD(\mathcal{D})=AD[0,1]$ contains all projections from $\mathcal{D}=L_\infty[0,1]$ (see Proposition  \ref{ad01}) and is, therefore, contained in the closure in measure topology of the set of all linear combinations of all projections. Since $\partial_{AD}$ vanishes on projections  and is continuous in the measure topology, it follows that  $\partial_{AD}=0$,
 and therefore, $\delta_n=0$. Hence, $\delta=0,$  which is a contradiction.
\end{proof}

\begin{rem}
 We note that $AD(\mathcal{R})$ contains a regular ring of continuous geometry for $\mathbb{C},$ namely,
$\mathcal{R}_\infty=\overline{\bigcup_{n\geq0}\mathcal{R}_n}^{\rho}$ is a continuous geometry
 by \cite[Theorems~D~and~E]{Neu}. Furthermore, the ring   $\mathcal{R}_\infty$ is a proper subalgebra of
$AD(\mathcal{R}).$
Indeed, $\delta_{AD}|_{\mathcal{R}_\infty}=0,$ because  by Proposition~\ref{tower lemma} $\delta_{AD}|_{\mathcal{R}_n}=0$ for all
$n\geq 0$ and $\delta_{AD}$ is $\rho$-continuous by  Proposition \ref{delta continuity lemma}. On the other hand, $\delta_{AD}$ is non trivial. Thus
$\mathcal{R}_\infty$ is a proper subalgebra of
$AD(\mathcal{R}).$

\end{rem}

\section{An example of a derivation on a Cartan masa which does not extend to $S(\mathcal{M})$}\label{sec5}

In this section, we prove that there is a derivation on the algebra $\mathcal{D}$ with values in $S(\mathcal{D})$, which can not be extended up to a derivation on $S(\mathcal{R})$. Using Connes--Feldmann--Weiss theorem we also prove analogous result for any Cartan masa in a hyperfinite type II$_1$ factor $\mathcal{R}$.

Let, as before, $\mathcal{D}$ be the ``diagonal'' masa in $\mathcal{R}$ (see Section \ref{nc_approx_differentiability}).
 As before, we identify $\mathcal{D}=L_\infty[0,1]$ and view $S(\mathcal{D})=S[0,1]$ as a $*$-subalgebra of $S(\mathcal{R})$.

\begin{thm}\label{thm_nonext_diagonal}Let  $\delta:AD(\mathcal{D})\to S(\mathcal{D})$ be a derivation, such that
$\delta|_{AD[0,1/2]}=\partial_{AD}$ and $\delta|_{AD[1/2,1]}=-\partial_{AD}$. The derivation $\delta$ cannot be extended up to a derivation from $S(\mathcal{R})$ to $S(\mathcal{R})$.

\end{thm}

\begin{proof}

 Let    $\pi: L_\infty[0,1]\longrightarrow\mathcal{D}$ be  a $\ast$-isomorphism defined in Section~\ref{nc_approx_differentiability}.

Let $\partial_{AD}:AD[0,1]\to S[0,1]$ be the approximate derivation. By Proposition \ref{prop_partial_AD_nonexpansive} this derivation is non-expansive. Therefore, by \cite[Theorem 3.1]{BSCh06}  there exists a derivation $\delta_0:S[0,1]\to S[0,1]$, which extends $\partial_{AD}$.

Denote for brevity, the first Rademacher function by $r=\chi_{(0,\frac12)}-\chi_{(\frac12,1)}$ and consider the mapping $\tilde{\delta}:S[0,1]\to S[0,1]$ defined by
$$\tilde{\delta}(x)=\delta_0(xr),\quad x\in S[0,1].$$
Since $\delta_0(r)=\partial_{AD}(r^2)=0,$ it follows that $\tilde{\delta}$ is a derivation on $S[0,1]$.

We set
$$\delta=\pi\circ\tilde{\delta}\circ\pi^{-1}: S(\mathcal{D})\longrightarrow S(\mathcal{D}).$$
We claim that $\delta$ is a derivation, which cannot be extended up to a derivation from $\mathcal{R}$ to $S(\mathcal{R})$.
Assume, by contradiction, that the derivation $\delta$ extends up to a derivation $D:\mathcal{R}\to S(\mathcal{R}).$

Consider the automorphisms $\gamma\in Aut(L_\infty[0,1])$ defined by setting
$$
(\gamma(x))(t)=x(\{t+1/2\}),\quad x\in L_\infty[0,1],
$$
where  $\{t\}$ is the fractional part of a number $t\in\mathbb{R}.$
Since $\gamma(\chi_{_{X_1}})=1-\chi_{_{X_1}}$ and $\gamma(\chi_{_{X_k}})=\chi_{_{X_k}}$ for  $k>1,$ we obtain that
$$
\pi(\gamma(x))=u\pi(x)u,\quad x\in S[0,1],
$$
where $$u=\begin{pmatrix}0&1\\1&0\end{pmatrix}\otimes  \left(\bigotimes\limits_{i=2}^\infty 1_i\right).$$

Let $f\in L_\infty[0,1]$ be given by $f(t)=t$.
Define the self-adjoint element $x_0\in \mathcal{D}$ by setting
$$x_0=\frac14+\pi(f\cdot r).$$
It is clear that,
$$\delta(x_0)=\pi(\tilde{\delta}(f\cdot r))=\pi(\delta_0(f))=1$$ and
$\gamma(x_0)=-x_0,$ because $\gamma(f)=f+\frac{1}{2}\chi_{(0,\frac12)}-\frac{1}{2}\chi_{(\frac12, 1)}$ and $\gamma(r)=-r.$
In particular, we have $ux_0u=-x_0$ and $ux_0=-x_0u.$

By the Leibniz rule, we have
$$D(ux_0u)=D(u)x_0u+uD(x_0)u+ux_0D(u)=D(u)x_0u+1+ux_0D(u).$$
Since $ux_0u=-x_0,$ it follows that
$$-1=D(-x_0)=-D(ux_0u)=D(u)x_0u+1+ux_0D(u).$$
Taking into account that $ux_0=-x_0u,$ we obtain
$$-2=-D(u)ux_0-x_0uD(u).$$
Since
$$D(u)u+uD(u)=D(uu)=D(1)=0,$$
it follows that
$$-2=uD(u)x_0-x_0uD(u)=[uD(u),x_0].$$
Since $x_0$ is self-adjoint, the latter equality contradicts Theorem~\ref{com}~(c).

Thus, $\delta:AD(\mathcal{D})\to S(\mathcal{D})$ cannot be extended up to a derivation from $S(\mathcal{R})$ to $S(\mathcal{R})$.
\end{proof}

We now prove a result similar to Theorem \ref{thm_nonext_diagonal} for an arbitrary Cartan masa in $\mathcal{R}$.

\begin{thm}\label{nonext thm} Let $\mathcal{A}$ be a Cartan masa in the hyperfinite II$_1-$factor $\mathcal{R}.$ There exists a derivation $\delta:\mathcal{A}\to S(\mathcal{A})$ which cannot be extended as a derivation from $\mathcal{R}$ to $S(\mathcal{R}).$
\end{thm}

\begin{proof}

By Connes--Feldmann-Weiss Theorem  \cite[Corollary 11]{CFW81}, there is an $*$-automorphism  $\alpha\in Aut(\mathcal{R})$ such that  $\alpha(\mathcal{D})=\mathcal{A}$.
Since any $*$-automorphism on $\mathcal{R}$ preserves the trace, it follows that the $*$-automorphism $\alpha$ uniquely extends to a continuous in the measure topology $*$-automorphism of the Murray--von Neumann algebra $S(\mathcal{R}),$ which we still denote by $\alpha $ (see Proposition \ref{prop_trace_pres_isom}).

%\footnote{For a type II$_1$-factor $\mathcal{M}$
%acting on a Hilbert space $H,$ an extension of  $\ast$-automorphism $\alpha$ can be obtained by the following way.
%
%Since $\mathcal{M}$ is a factor of type II$_1,$ by \cite[Exercise 9.27]{KRII} there exists a unitary $u\in B(H)$ such that
%$\alpha(x)=uxu^\ast$ for all $x\in \mathcal{M}.$ Let $x=x^\ast\in S(\mathcal{M})$ and let $x=\int\limits_{-\infty}^{+\infty} \lambda \, de_\lambda$ be its spectral resolution. Then
%$uxu^\ast=\int\limits_{-\infty}^{+\infty} \lambda \, d(ue_\lambda u^\ast)\in S(\mathcal{M}),$
%and therefore $u S(\mathcal{M})u^\ast\subseteq S(\mathcal{M}).$ Thus the mapping
%$
%\widetilde{\alpha}(x)=uxu^\ast,\, x\in S(\mathcal{M}),
%$
%is an $\ast$-automorphism of $S(\mathcal{M})$ which is an extension of $\alpha.$}
%Thus $\alpha(S(\mathcal{D}))=S(\mathcal{A}),$ because  $\mathcal{D}$ and $\mathcal{A}$ are dense in the measure topology in  $S(\mathcal{D})$ and $S(\mathcal{A}),$ respectively.

Now, let $\delta:\mathcal{D}\to S(\mathcal{D})$ be the derivation as in Theorem \ref{thm_nonext_diagonal}. Then the mapping  $\alpha\circ\delta\circ\alpha^{-1}:\mathcal{A}\to S(\mathcal{A})$  is well-defined and is a derivation. If $\alpha\circ\delta\circ\alpha^{-1}$ extends  to a derivation $D$ from $\mathcal{R}$ into $S(\mathcal{R})$, then  a derivation $\alpha^{-1}\circ D\circ\alpha:\mathcal{R} \to S(\mathcal{R})$ is an extension of $\delta$, which is not possible. Thus $\delta$ cannot be extended up to a derivation from $\mathcal{R}$ to $S(\mathcal{R})$.
\end{proof}

%
%\subsection{Some problems}
%
%
%
%
%
%
%In connection with Theorem~\ref{com} we have the following.
%
%\begin{prob}\label{factortwo}
%Let $\mathcal{M}$ be a von Neumann factor of type $II_1.$ Does there exists a derivation $D$ on $S(\mathcal{M})$ and a self-adjoint element
%$x\in \mathcal{M}$ such that $D(x)=\mathbf{1}$?
%\end{prob}
%
%
%Note that such derivations could not be continuous in the measure topology, because continuous in the measure topology
%derivations on $S(\mathcal{M})$ are inner  (see \cite{AK2013, BCS14}), {\color{red}  and by Theorem~\ref{com} (c) inner derivations do not  take value  $\mathbf{1}$ at self-adjoint elements.}
%
%

\end{document}